\documentclass{amsproc}
\usepackage{amsmath}
\usepackage{amssymb}
\usepackage{rotating}
\usepackage{graphicx}
\DeclareGraphicsExtensions{.png,.pdf,.eps}
\usepackage[all]{xy}
\usepackage{tikz}
\usepackage{enumitem}
\usepackage{array}
\usepackage[pdfstartview=FitH]{hyperref}

\newtheorem{theorem}[equation]{Theorem}
\newtheorem{lemma}[equation]{Lemma}
\newtheorem{corollary}[equation]{Corollary}
\newtheorem{proposition}[equation]{Proposition}

\theoremstyle{definition}

\newtheorem{definition}[equation]{Definition}

\newtheorem{example}[equation]{Example}
\newtheorem{algorithm}[equation]{Algorithm}

\numberwithin{equation}{section}

\def\bC{{\mathbb C}}

\def\bE{{\mathbb E}}

\def\bP{{\mathbb P}}
\def\bQ{{\mathbb Q}}
\def\bR{{\mathbb R}}
\def\bZ{{\mathbb Z}}
\def\bT{{\mathbb T}}

\def\cF{{\mathcal F}}

\def\cI{{\mathcal I}}

\def\cL{{\mathcal L}}

\def\cO{{\mathcal{O}}}
\def\cP{{\mathcal{P}}}

\def\cR{{\mathcal R}}

\def\F{\phi}

\def\pdiv{\operatorname{div}}

\def\H{\rm H}

\def\hM{\widehat{M}}
\def\hs{\widehat{\sigma}}

\def\iso{\cong}

\def\ra{\rightarrow}

\def\operatorname#1{\mathop{\rm #1}\nolimits}

\def\Proj{\operatorname{Proj}}

\def\Hom{\operatorname{Hom}}

\def\Pic{\operatorname{Pic}}

\def\Spec{\operatorname{Spec}}

\def\codim{\operatorname{codim}}

\def\deg{\operatorname{deg}}

\def\det{\operatorname{det}}
\def\mon{\operatorname{Mon}}

\def\Nef{{\operatorname{Nef}}}

\def\Mov{{\operatorname{Mov}}}
\def\Nu{{\operatorname{N_1}}}
\def\NU{{\operatorname{N^1}}}

\def\Eff{{\operatorname{Eff}}}

\def\Cl{\operatorname{Cl}}

\newcommand{\ignore}[1]{}

\definecolor{c5}{rgb}{0.1,0.45,0.03}
\definecolor{c2}{rgb}{0.6,0.3,0.1}
\definecolor{c3}{rgb}{0.6,0.1,0.3}
\definecolor{c4}{rgb}{0.5,0.35,0.15}
\definecolor{c1}{rgb}{0.5,0.15,0.35}
\newcommand{\pb}{\bullet}
\newcommand{\pbfour}{{\color{cyan}\bigstar}\normalcolor}
\newcommand{\pbtwo}{{\color{blue}\blacklozenge}\normalcolor}
\newcommand{\pquadric}{{\color{green}\blacksquare}\normalcolor}
\newcommand{\pbone}{{\color{c5}\blacktriangle}\normalcolor}
\newcommand{\pbthreecoll}{{\color{blue}\spadesuit}\normalcolor}

\makeindex

\begin{document}

\title[81 resolutions]{On 81 symplectic resolutions of a 4-dimensional
  quotient\\ by a group of order 32} \today \thanks{This research was
  conducted within the framework of the Polish National Science Center
  project 2013/08/A/ST1/00804 with support from grants
  2012/07/N/ST1/03202 (first author) and 2012/07/B/ST1/0334 (second
  author).}
\author[M.~Donten-Bury]{Maria Donten-Bury}
\address{Instytut Matematyki UW, Banacha 2, 02-097 Warszawa, Poland}
\email{M.Donten@mimuw.edu.pl, J.Wisniewski@uw.edu.pl}

\author[J.~A.~Wi\'sniewski]{Jaros\l{}aw A. Wi\'sniewski}



\begin{abstract}
  We provide a construction of 81 symplectic resolutions of a
  4-dimensional quotient singularity obtained by an action of a group
  of order 32. The existence of such resolutions is known by a result
  of Bellamy and Schedler. Our explicit construction is obtained via
  GIT quotient of the spectrum of a ring graded in the Picard group
  generated by the divisors associated to the conjugacy classes of
  symplectic reflections of the group in question. As the result we
  infer the geometric structure of these resolutions and their flops.
  Moreover, we represent the group in question as a group of
  automorphisms of an abelian 4-fold so that the resulting quotient
  has singularities with symplectic resolutions. This yields a new
  Kummer-type symplectic 4-fold.
\end{abstract}

\maketitle

\section{Introduction}

\subsection{Background}
Long before the notion of the Mori Dream Space was brought to life by
Hu and Keel in \cite{HuKeel} and put in the context of the Mumford's
Geometric Invariant Theory (GIT), \cite{MumfordGIT}, and related to
homogeneous coordinate rings, which were introduced for toric
varieties by David Cox in \cite{Cox}, a similar concept emerged in the
local study of contractions in the Minimal Model Program (MMP). In
particular, in his 1992 inspiring talk \cite{ReidFlip} Miles Reid
explained how to view flips in terms of variation of GIT. Although the
notion of the total coordinate ring, known also as the Cox ring, has
been extensively studied for projective varieties in the last decade,
see for example \cite{CastravetTevelev}, \cite{StillmanTestaVelasco},
\cite{ArtebaniHausenLaface}, apparently it has not been used for
understanding local contractions nor for resolution of higher
dimensional singularities, as it was proposed in \cite{ReidFlip}.

In the present paper we follow the ideas of \cite{ReidFlip} and
construct a family of crepant resulutions of a symplectic singularity
via GIT quotients. Two preceding papers which used similar ideas,
\cite{FG-AL}, \cite{DontenSurfaces}, concerned the case of surface
singularities for which the resolutions are classically known. The
excellent book \cite{CoxRings} which provides the exhaustive overview
of the present state of knowledge about Cox rings and their
applications tackles this problem for toric and complexity-1 case.

In the present paper we focus on a 4-dimensional symplectic
singularity which is known to admit such a resolution by Bellamy and
Schedler, \cite{BellamySchedler}. The result of Bellamy and Schedler
is based on relation of symplectic resolutions to smoothings of the
singularity by a Poisson deformation, see \cite{NamikawaPoisson},
\cite{GinzburgKaledin}.  The methods which we provide in the present
paper are of completely different nature. They reveal an explicit
description of the resolutions, which is a significant advantage over
the previous approach which was non-effective in this respect.

\subsection{Resolutions via GIT quotients}

The main result of the present paper is the following.
\begin{theorem}\label{main_theorem}
  Let $V$ be a 4-dimensional vector space with a symplectic form
  $\omega$. Assume that $G$ is a group of order 32 defined in Section
  \ref{the-group}, acting on $V$ and preserving $\omega$. By $\bT$ we
  denote a 5-dimensional algebraic torus with coordinates $t_i$,
  $i=0,\dots 4$ associated to 5 classes of symplectic reflections
  generating $G$.\par
  Let $\cR$ be a $\bC$-subalgebra generated in $\bC[V]\otimes\bC[\bT]$
  by $t_i^{-2}$ for $i=0,\dots, 4$ and $\F_{ij}t_it_j$ for $0\leq
  i<j\leq 4$, where $\F_{ij}$ are eigen-functions of the action of
  $Ab(G)$ on $\bC[V]^{[G,G]}$, as defined in Lemma
  \ref{eigenvectors} \par
  Then there are 81 GIT quotients of $\Spec\cR$ which yield all crepant
  resolutions of $V/G$. A distinguished resolution of $V/G$ has a
  2-dimensional fiber which is a union of a $\bP^2$ blown-up in 4
  points and of ten copies of $\bP^2$.
\end{theorem}

The same number of different symplectic resolutions of $V/G$ has been
calculated by Bellamy, \cite{Bellamy81}, by using the results of
Namikawa, \cite{Namikawa2}. Recently, using methods from
\cite{HausenMoscow} and \cite{HausenKeicherLaface}, Hausen and Keicher
have proved that the ring $\cR$ in our theorem is, in fact, the total
coordinate ring of any resolution of $V/G$, \cite{HausenPrivate}.

Our construction of the ring $\cR$ is built up on the base of the
total coordinate ring of the quotient $V/G$ which is equal to
$\bC[V]^{[G,G]}\subset \bC[V]$, \cite{ArzhantsevGauifullin}.  The
generators of $\cR$ are the classes of exceptional divisors of the
resolution and strict transforms of Weil divisors associated to
homogeneous generators of $\bC[V]^{[G,G]}$. In our theorem these are
the generators of the type $t_i^{-2}$ and $\F_{ij}t_it_j$,
respectively. We note that by the McKay correspondence,
\cite{KaledinMcKay}, the exceptional divisors are in correspondence
with classes of symplectic reflections in $G$.

\subsection{Contents of the paper}
In Section \ref{prelim} we introduce definitions and recall results
which are needed in subsequent sections.  Next, in Section \ref{total}
we introduce the layout for constructing the total coordinate ring of
resolutions of quotient singularities.

The idea is as follows. It is known that the Cox ring of a quotient
singularity $V/G$ is the ring of invariants of the commutator
$\bC[V]^{[G,G]}$ which decomposes into eigenspaces of the action of
the abelianization $Ab(G)$ which can be identified with the class
group $\Cl(V/G)$. That is $\bC[V]^{[G,G]}=\bigoplus_{\mu\in
  G^\vee}\bC[V]^G_\mu$ where $\bC[V]^G_\mu$ is a rank 1 reflexive
$\bC[V]^G$ module associated to a character $\mu\in G^\vee$. Given a
resolution $\varphi: X\ra V/G$ the push-forward map allows to identify
spaces of sections of line bundles over $X$ with submodules of these
eigen-modules, $\Gamma(X,\cO_X(D)) \hookrightarrow
\Gamma(V/G,\cO_{V/G}(\varphi_*D))$.

If $D=\sum_i a_iE_i$, where $E_i$'s are exceptional divisors of
$\varphi$ and $a_i$'s are integers, then it follows that
$\Gamma(X,\cO_X(D)) = \{f\in\Gamma(V/G,\cO_{V/G}):\ \forall_i\ \
\nu_{E_i}(f)\geq -a_i\}$ where $\nu_{E_i}$'s are divisorial valuations on
the field $\bC(X)=\bC(V)^G$. This observation does not make much sense
for a general $D$ but, fortunately, in case of symplectic resolutions
each valuation $\nu_{E_i}$ can be related to a monomial valuation
$\nu_{T_i}$ on the field of fractions of $\bC[V]$ which comes from the
action of the symplectic reflection $T_i$ of $V$ associated to $E_i$
via the McKay correspondence. We use this idea in to construct the
ring $\cR$ in Theorem \ref{main_theorem}.

The resolution $X\ra V/G$ is recovered as a GIT quotient of $\Spec\cR$
with respect to the action of the algebraic torus $\bT_{\Cl(X)}$ which
is associated to grading of $\cR$ in the class group $\Cl(X)$. We do
it in Section \ref{section_proofs}. Although the construction of a GIT
quotient and verification of its smoothness is, in general,
conceptually clear it is still computationally involved. Firstly, we
find out that a natural choice of a character $\kappa$ of
$\bT_{\Cl(X)}$ yields a good GIT quotient, see subsection
\ref{section_irrelevant_ideal}. Secondly, we verify that the resulting
GIT quotient is indeed smooth, see subsection
\ref{section_singularities}. To perform the calculations we use an
embedding of $\Spec\cR$ in an affine space, in which the action of
$\bT_{\Cl(X)}$ is diagonal, see subsection
\ref{section-generators}. Next, we consider a covering of $\Spec\cR$
by sets associated to orbits of the big torus of the affine
space. After all reductions and using symmetries we are left with a
manageable computational problem which is calculated and
cross-calculated with standard algebraic software: \cite{M2},
\cite{Singular} and \cite{sage}.

A striking consequence of our calculation of the resolution
$\varphi^\kappa: X^\kappa\ra V/G$ is that it has the unique 2-dimensional
fiber containing, as a component, the blow-up of $\bP^2$ at four general
points. In terms of McKay correspondence this component is related to
the only nontrivial central element of the group $G$. The Picard group
of $X^\kappa$ can be identified with the Picard group of this
surface. In fact, different resolutions $\varphi: X\ra V/G$ are in
relation to Zariski decomposition of divisors on that surface; we
describe the geometry of all resolutions and their flops, see Section
\ref{geom}.

In the final section of the paper we follow the program initiated in
\cite{AW-Kummer} and \cite{DontenKummer}. We produce a representation
of the group $G$ in the group of automorphisms of an abelian 4-fold
which is the product of 4 elliptic curves with complex
multiplication. The resulting quotient admits a symplectic resolution
hence we obtain a new Kummer-type symplectic 4-fold,
\ref{Kummer4fold-3}. Thus, the second main result of our paper is the
following
\begin{theorem}\label{2nd_main_theorem}
  Let $\bE$ be an elliptic curve with the complex multiplication by
  $i=\sqrt{-1}$ and $G$ the group defined in \ref{the-group}.  Then
  there exists an embedding $G\ra G'\subseteq Aut(\bE^4)$ such that
  the quotient $\bE^4/G'$ has a resolution which is a Kummer
  symplectic 4-fold $X$ with $b_2(X)=23$ and $b_4(X)=276$.
\end{theorem}

\subsection{Notation}
We use the standard notation in set and group theory.
A commutator of $G$ is denoted by $[G,G]$ and by $Ab(G)$ we denote its
abelianization $G/[G,G]$. A group of characters of $G$ is
$G^\vee=\Hom(G,\bC^*)$. The quotient group $\bZ/\langle r\rangle$ will
be denoted by $\bZ_r$ with $[d]_r$ denoting the class of $d\in\bZ$. By
$\lfloor d/r\rfloor$ we denote the integral part (round-down) of the
fraction $d/r$.  If $G$ acts on a set $B$ then by $B^G$ we denote the
set of the fixed points of the action. In particular if $G$ acts by
homomorphisms on a ring $B$ then by $B^G$ we denote the ring of
invariants of the action.

A torus $\bT$ means an algebraic torus with finite lattice of
characters (called also monomials) $M=M_\bT=\Hom_{alg}(\bT,\bC^*)$ and
dual lattice of 1-parameter subgroups $N=M^*$. Given a finitely
generated free abelian group $M$ we define associated torus
$\bT_M=\Hom(M,\bC^*)$ with lattice of characters equal $M$.  We drop
subscripts whenever it makes no confusion. The pairing $M\times
N\rightarrow \bZ$ is denoted by $(u,v)\mapsto \langle u,v\rangle$.

For $u\in M_\bT$ by $\chi^u$ we denote the character $\chi^u:
\bT\rightarrow\bC^*$. By $\bC[M_\bT]$ we denote the ring of Laurent
polynomials graded in the lattice $M_\bT$.  More generally, for a
cone $\sigma\subset M_\bQ$ by $\bC[M\cap\sigma]$ we understand the
respective subalgebra of $\bC[M]$.  For the lattice $N$ or $M$ with a
given basis by $\sigma^+_N$ or, respectively, $\sigma^+_M$ we will
denote the convex cone generated by the basis, to which we refer as
positive orthant. Thus, for a $\bC$-linear space $V$ we have
$\bC[V]\subset\bC[M_\bT]$ where $\bC[V]$ is the ring of polynomials in
linear coordinates of $V$ and $\bT=\bT_V$ is the standard torus acting
diagonally on coordinates of $V$.

All varieties are defined over the field of complex numbers. By
$\bC(X)$ we denote the field of rational functions on a variety $X$
while by $\bC[X]$ we denote the algebra of global functions on $X$.
By $\Spec A$ we understand the maximal spectrum of a ring $A$; in
particular if $X$ is an affine variety then $X=\Spec\bC[X]$.

For a $\bQ$-factorial variety $X$ by $\NU(X)$ and $\Nu(X)$ we will
denote the $\bR$-linear space of divisors and, respectively, of proper
1-cycles on $X$, modulo numerical equivalence. The class of a divisor
$D$ or a curve $C$ will be denoted by $[D]$ and $[C]$,
respectively. We have a natural intersection pairing of these two
spaces.

A cone in a $\bR$-vector space $N$ generated by elements $v_1,
v_2\dots$ is, by definition, $cone(v_1,v_2\dots)=\sum_i\bR_{\geq
  0}v_i$.  By $\Nef(X)\subset\NU(X)$ we denote the cone of divisors
which are non-negative on effective 1-cycles. By $\Eff(X)$ and
$\Mov(X)$ we denote the cones in $\NU(X)$ which are $\bR_{\geq
  0}$-spanned by the classes of effective divisors and, respectively,
by divisors whose linear systems have no fixed component.

A blow-up of $\bP^2$ in $r$ general points will be denoted by
$\bP^2_r$. The blow-up of $\bP^2$ in three (different) collinear
points will be denoted by $\bP^2_{\overline{3}}$.  

\par\medskip

\noindent {\bf Acknowledgments.} We would like to thank J\"urgen
Hausen for informing us about \cite{HausenPrivate} and Maksymilian
Grab for pointing out a mistake in a previous version of this paper.
We also thank the referee for a thorough review of this paper.


\section{Preliminaries}\label{prelim}
\subsection{The total coordinate ring}\label{CoxRing}
Let $X$ be a normal $\bQ$-factorial variety over the field of complex
numbers. We assume that $\bC[X]$ is finitely generated, its
invertible elements are in $\bC^*$, and $X$ is projective (hence
proper) over $\Spec\bC[X]$.
In what follows we also assume that the divisor class group $\Cl(X)$
is finitely generated.

In order to define the total coordinate ring of $X$, called also the
Cox ring of $X$, we set
\begin{equation}\label{Coxring-def}
\cR(X)=\bigoplus_{[D]\in\Cl(X)}\Gamma(X,\cO_X(D))
\end{equation}
where $\Gamma(X,\cO_X(D))$ denotes the space of global sections of the
reflexive sheaf associated to the linear equivalence class of the
divisor $D$. It is standard to use the identification
$\Gamma(X,\cO_X(D))= \{f\in\bC(X)^*: \pdiv(f)+D\geq 0\}\cup\{0\}$ and
to say that a non-zero $f\in\Gamma(X,\cO_X(D))$ is associated with an
effective divisor $D'=\pdiv(f)+D$. The divisor $D'$ is the zero locus
of this section of $\cO_X(D)$ which we will usually denote by
$f_{D'}$. Note that by our assumptions the relation $D'\leftrightarrow
f_{D'}$ is unique up to multiplication by a constant from $\bC^*$.

Now, to define the multiplication in $\cR(X)$ properly we
need to set the inclusion $\Gamma(X,\cO_X(D))\subset \bC(X)$ so
that it does not depend on the choice of the divisor $D$ in its linear
equivalence class.  If $\Cl(X)$ is a free finitely generated group (no
torsions) then we choose divisors $D_1,\dots, D_m$ whose classes make
a basis of $\Cl(X)$ and for $D$ linearly equivalent to $\sum_i
a_iD_i$, with $a_i\in\bZ$, we define
\begin{equation}\label{sections-basis}
  \begin{array}{ll}
    \Gamma(X,\cO_X(D))&=\left\{f\in\bC(X)^*:\ \pdiv(f)+\sum_i a_iD_i\geq 0\right\}\cup\{0\}\\
    &= \left\{f\in\cO_X(X\setminus\bigcup_iD_i)\subset\bC(X):
      \forall_i\ \  \nu_{D_i}(f)\geq -a_i\right\}
  \end{array}
\end{equation}
where the second equality makes sense if $D_i$'s are prime divisors
and for every $i$ we take $\nu_{D_i}: \bC(X)\ra\bZ\cup\{\infty\}$, a
valuation centered at the divisor $D_i$. This way we present each
graded piece of $\cR(X)$ as a $\bC[X]$-sub-module of $\bC(X)$ and
consequently we define the multiplication in $\cR(X)$ as inherited
from $\bC(X)$. It can be checked that this defines a ring structure on
$\cR(X)$ which does not depend on the choice of $D_i$'s (different
choices give isomorphic rings).  This definition has to be adjusted if
$\Cl(X)$ has torsions. We advise the reader to consult
\cite[Ch.~1]{CoxRings} for details.

We recall that, for a Weil divisor $D$ on a normal variety $X$, the
sheaf $\cO_X(D)$ is reflexive of rank one. The association
$D\ra\cO_X(D)$ determines the bijection between the class group
$\Cl(X)$ and the isomorphism classes of reflexive rank-1 sheaves over
$X$. Locally, in a similar manner, we can define a graded sheaf of
divisorial algebras over $X$ which we will denote by $\cR_X$.

Suppose that $\cR(X)$ is a finitely generated $\bC$-algebra. Then
$\Spec\cR(X)$ is an affine variety and the grading of $\cR(X)$ in
$\Cl(X)$ determines the action of the algebraic quasi-torus
$\bT_{\Cl(X)}=\Hom(\Cl(X),\bC^*)$.

As we assumed that $X$ is projective over $\Spec\bC[X]$, we can use
geometric invariant theory, or GIT, \cite{MumfordGIT}, to recover $X$
as a quotient of $\Spec\cR(X)$. Namely, a relatively ample divisor
$H$ on $X$ determines a linearization of this action on the trivial
line bundle over $\Spec\cR(X)$ and the irrelevant ideal $\cI
rr_H=\sqrt{(\Gamma(X, \cO(mH)): m>0)}$ which determines the set of
unstable points with respect to this linearization. The quotient of
its complement is $X$.  The choice of a big but not necessarily ample
divisor on $X$ yields a GIT quotient which is birational to $X$.

This line of argument works in much broader set up as it is explained
in \cite[Ch.~1]{CoxRings}. In the present paper we will concentrate on
a special situation of quotient singularity and its resolution which
will be discussed in detail in subsequent sections.

\subsection{Small blow-ups of $\bP^2$ }\label{P^2_4}
The surface obtained by blowing up $r$ general points on $\bP^2$ will
be denoted by $\bP^2_r$. If $r\leq 3$ then $\bP^2_r$ is a toric
surface whose geometry and Cox ring are well known.  Here, for the sake
of completeness, we recall properties of $\bP^2_4$.

If $\bP^2_4$ is obtained by blowing $\bP^2$ at points $p_1,\dots,p_4$
then by $F_{0i}$ we denote the exceptional $(-1)$-curve over $p_i$ and
by $F_{ij}$, with $1\leq i< j \leq 4$, we denote strict transform of
the line passing through $\{p_1\dots,p_4\}\setminus\{p_i,p_j\}$. With
this notation we have the following well known facts, see
e.g.~\cite{ManinCubicForms}.

\begin{lemma}\label{geom-P^2_4}
The surface $\bP^2_4$ has the following geometry:
\begin{enumerate}
\item For two pairs of numbers $0\leq i<j\leq 4$ and $0\leq p<q\leq 4$
  we have $F_{ij}\cdot F_{pq}= |\{i,j\}\cup\{p,q\}|-3$.
\item The surface $\bP^2_4$ admits five distinct birational morphisms
  $\beta_i: \bP^2_4\ra \bP^2$ such that $\beta_i$ contracts four
  $(-1)$-curves $F_{pq}$ for $i\in\{p,q\}$.
\item The surface $\bP^2_4$ admits five conic fibrations
  $\alpha_i:\bP^2_4\ra\bP^1$, each of them having three reducible
  fibers $F_{rs}\cup F_{pq}$, where $\{i,p,q,r,s\}=\{0,\dots,4\}$.
\end{enumerate}
\end{lemma}

Let us consider a 5-dimensional $\bR$-vector space $W$ with a basis
$e_0,\dots,e_4$. We define the intersection product on $W$ by setting
$e_i^2=-3$ and $e_i\cdot e_j=1$, if $i\ne j$.

For $0\leq i<j\leq 4$ we set $f_{ij}=(e_i+e_j)/2$. If moreover we set
$\kappa=\sum_i e_i$ and $c_i=[\kappa-e_i]/2=[(\sum_j e_j)-e_i]/2$,
then $\kappa\cdot e_i=\kappa\cdot f_{ij}=1$, $\kappa^2=5$, and
$f_{ij}^2=-1$. Also $e_i\cdot c_i=2$ and $e_i\cdot c_j=0$ for $i\ne j$
hence the base $e_i/2$, for $i=0,\dots,4$ is dual, in terms of the
intersection product, to the base consisting of $c_i$'s.  In
particular, $cone(e_0,\dots, e_4)$ is where the intersection with
$c_i$'s is positive.

By $\Lambda\subset W$ we denote the lattice spanned by $f_{ij}$'s. We
note that $\sum_i a_i(e_i/2)\in\Lambda$ if $a_i$'s are integral and
$\sum_i a_i$ is even. The following can be easily verified.

\begin{lemma}\label{Nef-Eff-P^2_4}
  The space $\NU(\bP^2_4)=\Nu(\bP^2_4)$ with intersection product and
  the lattice of integral divisors $\Pic(\bP^2_4)$ can be identified
  with $W\supset\Lambda$, so that $[F_{ij}]=f_{ij}$ and
  $[-K_{\bP^2_4}]=\kappa$.  Under this identification the cone
  $\Eff(\bP^2_4)$ is spanned by $f_{ij}$'s. It has 10 facets:
  \begin{itemize}
  \item five facets of $\Eff(\bP^2_4)$ are associated to morphisms
    $\alpha_i$ from \ref{geom-P^2_4} and they are contained in the
    facets of $\sigma^+$; given $i\in\{0,\dots,4\}$ a facet of this
    type is spanned by six $f_{pq}$'s such that $i\not\in\{p,q\}$
    and it is perpendicular (in the sense of the intersection product)
    to $c_i=(\kappa-e_i)/2$;
  \item five simplicial facets of $\Eff(\bP^2_4)$ are associated to
    morphisms $\beta_i$ from \ref{geom-P^2_4} and they are obtained by
    cutting $\sigma^+$ with a hyperplane perpendicular to
    $(\kappa+e_i)/2$; this facet of $\Eff(\bP^2_4)$ is spanned by four
    $f_{pq}$'s such that $i\in\{p,q\}$.
  \end{itemize}
  As the consequence, the cone $\Nef(\bP^2_4)$ is spanned by the
  classes of $(\kappa\pm e_i)$'s.
\end{lemma}

The cone $\Eff(\bP^2_4)$ is divided into Zariski chambers depending on
the Zariski decomposition of the divisors whose classes are inside the
interior of each chamber, see \cite{BauerFunkeNeumann}. For example,
the ``central'' chamber is the cone $\Nef(\bP^2_4)$ and if $[D]$ is
in the interior of $\Nef(\bP^2_4)$ then the linear system $|mD|$, $m\gg
0$ determines a morphism into the projective space whose image is
$\bP^2_4$. Equivalently, $\Proj\bigoplus_{m\geq 0}
\Gamma(\bP^2_4,\cO(mD)) = \bP^2_4$. For $D$ outside the nef cone the
image of the rational map defined by $|mD|$ (or this projective
spectrum) will depend on the intersection of $D$ with $(-1)$ curves
$F_{ij}$. This determines the division of $\Eff(\bP^2_4)$ into the
chambers in question. We summarize the information in Table
\ref{ZariskiP^2_4}, see e.g.~\cite{BauerFunkeNeumann}\par\medskip

\begin{table}[h]\caption{Zariski chambers and birational images of
    $\bP^2_4$}\label{ZariskiP^2_4}
\begin{tabular}{l|l}
  $\Proj(\bigoplus_{m\geq 0}\Gamma(\bP^2_4,\cO(mD))$&
  number of chambers with $D$ of this type\\
  \hline
  $\bP^2_4$&one, nef cone\\
  $\bP^2_3$&ten\\
  $\bP^2_2$&thirty\\
  $\bP^2_1$&twenty\\
  $\bP^2$&five, associated to simplicial facets of $\Eff(\bP^2_4)$\\
  $\bP^1\times\bP^1$&ten\\
\end{tabular}
\end{table}
\par\medskip
The following result is known, see e.g.~\cite{StillmanTestaVelasco}.

\begin{proposition}\label{Cox-P^2_4}
  The total coordinate ring of $\bP^2_4$ coincides with the projective
  coordinate ring of the Grassmann variety $Gr(2,W)$ of planes in a 5
  dimensional vector space $W$ which is embedded in $\bP(\bigwedge^2
  W^*)$ via the Pl\"ucker embedding. That is, $\cR(\bP^2_4)$ is the
  quotient of the polynomial ring in variables $w_{ij}$, for $0\leq
  i<j\leq 4$, by the ideal generated by the following quadratic
  trinomials
  $$\begin{array}{cc}
    w_{14}w_{23}+w_{13}w_{24}-w_{12}w_{34}&
    w_{04}w_{23}-w_{03}w_{24}-w_{02}w_{34}\\
    w_{04}w_{13}+w_{03}w_{14}-w_{01}w_{34}&
    w_{04}w_{12}-w_{02}w_{14}-w_{01}w_{24}\\
    w_{03}w_{12}+w_{02}w_{13}-w_{01}w_{23}
  \end{array}$$

  The action of the Picard torus can be identified with that of
  $\bT_\Lambda=\bT_W/\langle -I_W\rangle$, where $\bT_W$ is the
  standard torus of the space $W$ with which acts on $\bigwedge^2 W$
  with isotropy $\langle -I_W\rangle$. In other words, it is given by
  a grading in $\Lambda\subset\bZ^5$ such that $\deg
  w_{01}=(1,1,0,0,0),\ \deg w_{02}=(1,0,1,0,0),\dots,\ \deg
  w_{34}=(0,0,0,1,1)$.

  The GIT quotients of the affine variety $\Spec \cR(\bP^2_4)$ depend
  on the choice of the linearization of the torus action given by a
  character in the lattice $\Lambda$ . In particular, the choice of a
  divisor class in the interior of each of the Zariski chambers of
  $\bP^2_4$ determines the quotient as in Table \ref{ZariskiP^2_4}.
\end{proposition}

\subsection{The group of order 32}\label{the-group}
In what follows we will consider complex matrices acting on a linear
space $V$. The case of our primary interest is when $V$ is of even
dimension and admits a non-degenerate linear 2-form, which we will
call a symplectic form. The group of linear transformations preserving
such a form we will denote by $Sp(V)$ or $Sp(\dim\! V, \bC)$.

By $I_V$ we denote the identity matrix and by $-I_V$ its opposite. We
will usually drop the subscript. Also, for every matrix $A$ by $-A$ we
denote its opposite, that is $-I\cdot A=A\cdot(-I)$. For two matrices
$A$ and $B$ we set $[A,B]=A\cdot B\cdot A^{-1}\cdot B^{-1}$. If the
linear space fixed by a matrix is of codimension 1 then we say that it
is a quasi-reflection. The groups which we will consider contain no
quasi-reflections. If $A\in Sp(V)$ and its fixed point set is of
codimension 2 then we call it a symplectic reflection.

For $V$ of dimension 4 with coordinates $x_1,\dots, x_4$ we consider
the following matrices in $GL(V)$.
\begin{equation}\label{matrices}
\begin{array}{cc}
  T_0=\left(\begin{array}{cccc}1&0&0&0\\0&-1&0&0\\0&0&1&0\\0&0&0&-1\end{array}\right)&
  T_1=\left(\begin{array}{cccc}0&i&0&0\\-i&0&0&0\\0&0&0&-i\\0&0&i&0\end{array}\right)\\ \\
  T_2=\left(\begin{array}{cccc}0&1&0&0\\1&0&0&0\\0&0&0&1\\0&0&1&0\end{array}\right)&
  T_3=\left(\begin{array}{cccc}0&0&0&1\\0&0&-1&0\\0&-1&0&0\\1&0&0&0\end{array}\right)\\
\end{array}
\end{equation}
$$
\begin{array}{c}
  T_4=\left(\begin{array}{cccc}0&0&0&i\\0&0&-i&0\\0&i&0&0\\-i&0&0&0\end{array}\right)
\end{array}
$$

We note that the symplectic form $\omega=dx_1\wedge dx_3 + dx_2\wedge
dx_4$ is preserved by each of the $T_i$'s. Moreover, they are
symplectic reflections of order two.

We list the following facts which can be verified easily.
\begin{lemma}\label{matrices-properties} If, for $i\ne j$ we
  set $R_{ij}=T_i\cdot T_j$, then $R_{ij}$'s are of order 4 and
  moreover the following holds:
  \begin{enumerate}
  \item $[T_i,T_j]=R_{ij}^2=-I$, $T_0\cdot T_1\cdot T_2\cdot T_3\cdot
    T_4=I$, and $R_{ij}=-R_{ji}=R_{ji}^{-1}$
  \item $[T_s,R_{ij}]=-I$ if $s\in\{i,j\}$ and $[T_s,R_{ij}]=I$ if
    $s\not\in\{i,j\}$
  \item For two distinct pairs $i,j$ and $p,q$ we have
    $[R_{ij},R_{pq}] = -I$ if $\{i,j\}\cap\{p,q\}\ne\emptyset$ and
    $[R_{ij},R_{pq}] = I$ if $\{i,j\}\cap\{p,q\}=\emptyset$
  \end{enumerate}
\end{lemma}

Let $G$ be the group generated by the reflections $T_i$ in $Sp(V)$.
It is worthwhile to note that this group is conjugate in $Sp(V)$ to
the group generated by Dirac gamma matrices, c.f.~\cite{Temple}.

\begin{lemma}\label{group}
  The group $G$ is of order 32 and it has the following properties:
  \begin{enumerate}
  \item There are 17 classes of conjugacy of elements in $G$; two of
    the them consist of single elements $I$ and $-I$, five of them
    contain pairs of opposite reflections $\pm T_i$, for
    $i=0,\dots,4$, and ten conjugacy classes consist of pairs of
    opposite elements $\pm R_{ij}$, with $0\leq i< j\leq 4$.
  \item The commutator $[G,G]$ of $G$ coincides with its center and it
    is generated by $-I$, the abelianization is $Ab(G)= G/[G,G] =
    \bZ_2^4$.
  \item If $N(T_i)<G$ is the normalizer (or centralizer) of the
    reflection $T_i$ then $N(T_i)/\langle T_i\rangle \iso Q_8$, where
    $Q_8$ is the group of quaternions, or a binary-dihedral group, of
    order 8.
\end{enumerate}
\end{lemma}

\subsection{Quotient singularities}\label{sect-quotsing}

Let $G\subset GL(V)$ be a finite group with no quasi-reflections
acting linearly and faithfully on $V\iso\bC^n$. By
$\bC[V]\iso\bC[x_1,\dots, x_n]$ we understand the coordinate ring of
the linear space with the ring of invariants denoted by $\bC[V]^G$
which, by Hilbert-Noether theorem, is finitely generated
$\bC$-algebra. We set $Y=\Spec\bC[V]^G$.

\begin{proposition}\label{Cl_of_quotient}
  We have the isomorphisms: $\Pic(Y)=1$, $\Cl(Y)=G/[G,G]$ and
  $\cR(Y)=\bC[V]^{[G.G]}$ .
\end{proposition}
\begin{proof}
  The first part is classical nowadays, see e.g.~\cite{Benson}, the
  second part is in \cite{ArzhantsevGauifullin}.
\end{proof}

Recall that the abelianization $Ab(G)=G/[G,G]$ is isomorphic to the
group of characters $G^\vee=\Hom(G,\bC^*)$. The ring
$\cR(Y)=\bC[V]^{[G,G]}$, as $\bC[Y]$-module, is a direct sum of
reflexive rank one modules associated to characters of the group $G$,
see \cite[Thm.~1.3]{Stanley}. That is, the grading of $\cR(Y)$ is into
the eigenspaces associated to the characters of $G$:
\begin{equation}\label{Coxringdecomposition}
  \cR(Y)=\bC[V]^{[G,G]}=\bigoplus_{\mu\in G^\vee}\bC[V]^G_\mu\end{equation}
where $\bC[V]^G_\mu$ is a rank-1 reflexive $\bC[Y]=\bC[V]^G$ module on
which $G$ acts with character $\mu\in G^\vee$, for discussion
see \cite[Lemma 6.2]{DontenSurfaces}.

The following fact is in \cite[Sect.~2]{DontenSurfaces}.
\begin{lemma}\label{pic-resolution}
  Let $\varphi: X\ra Y$ be a resolution of a quotient
  singularity. Then $\Pic X=\Cl X$ is a free finitely generated
  abelian group.
\end{lemma}
\begin{proof} By Grauert-Riemenschneider we know that $\H^1(X,\cO)=0$
  hence it is enough to prove that $\Pic X =\H^2(X,\bZ)$ has no finite
  torsion. However by \cite[Thm.~7.8]{Kollar} we know that the
  fundamental group is trivial so, by universal coefficients theorem,
  $\H^2(X,\bZ)$ has no torsion.
\end{proof}

We note that the exceptional set of a resolution $\varphi: X\ra Y=V/G$
is a divisor with $m$ irreducible components $E_i$ and thus $\Cl X$ is
a free abelian group of rank $m$. Moreover we have the following exact
sequence
\begin{equation}\label{class-sequence}
  0\longrightarrow \bigoplus_{i=1}^m \bZ[E_i]\longrightarrow \Cl X\iso
  \bZ^m \longrightarrow \Cl Y = Ab(G)\longrightarrow 0
\end{equation}
where the right arrow is the push-forward morphism $\varphi_*$.

Now we assume that $\dim V=2n$ and $G$ is a finite group in $Sp(V)$,
and there exists a symplectic resolution $\varphi: X\rightarrow
Y=V/G$.  That is $X$ admits a closed everywhere non-degenerated 2-form
which restricts to the invariant symplectic form defined on the smooth
locus of $V/G$.  Then the morphism $\varphi: X\rightarrow Y$ is
semi-small which means, in particular, that every exceptional divisor
of $\varphi$ in $X$ is mapped to a codimension 2 component of the
singular set associated to the fixed point locus of a symplectic
reflection. In fact $G$ has to be generated by symplectic reflections
and we have the following version of McKay correspondence by Kaledin,
see \cite{KaledinMcKay}.

\begin{theorem}\label{KaledinMcKay}
  There exist a natural basis of Borel-Moore homology
  $\H_{\bullet}^c(X,\bQ)$ whose elements are in bijection with the
  following objects:
  \begin{itemize}
  \item irreducible closed subvarieties $Z\subset X$ for which it
    holds $2\codim_XZ=\codim_Y\varphi(Z)$;
  \item conjugacy classes of elements $g$ in $G$.
  \end{itemize}
  Under this correspondence every conjugacy class $[g]$ of an element
  $g\in G$ is related to $Z^{[g]}\subseteq X$ such that
  $\varphi(Z^{[g]}) = [V^g]$, where $[V^g]$ denotes the image in $V/G$
  of the linear subspace $V^g$ of fixed points of $g$.
\end{theorem}

Since the exceptional divisors $E_i$ of $\varphi$ are contracted to
codimension two set in $Y$ their configuration is modeled on
2-dimensional Du Val singularities. Namely, if $C_i\subset E_i$ is an
irreducible component of a general fiber of $\varphi_{|E_i}$ then
$C_i$ is a rational curve and the intersection matrix $(E_i\cdot C_j)$
is a direct sum of known Cartan-type matrices, see
\cite[Thm.~1.3]{Wierzba}, \cite[Thm.~4.1]{AW2014}. If $[L_i]\in
\NU(X)$, for $i=1,\dots, m$, is the basis dual (in terms of the
intersection) to that consisting of $[C_i]\in N_1(X)$, then $\Pic(X)$
is a sublattice of the lattice $\langle [L_1], \dots, [L_m]\rangle$
and we note that the left-hand side arrow in the sequence
\ref{class-sequence} can be written as follows $E_i \mapsto \sum_j
(E_i\cdot C_j)[L_j]$. Thus, the sequence \ref{class-sequence} can be
extended to the following diagram in which the left-hand side arrow in
the lower sequence is a map of lattices of the same rank given by the
intersection matrix $(E_i\cdot C_j)$
\begin{equation}\label{class-sequence-1}
\xymatrix{
0\ar[r]&\bigoplus_i \bZ[E_i]\ar[r]\ar@{=}[d]&
\Cl(X)\ar[r]\ar@{^{(}->}[d]&\Cl(V/G)\ar[r]\ar@{^{(}->}[d]&0\\
0\ar[r]&\bigoplus_i \bZ[E_i]\ar[r]&\bigoplus_i \bZ[L_i]\ar[r]
&Q\ar[r]&0
}
\end{equation}
Here $Q$ is the quotient of lattices $\bigoplus_i \bZ[E_i]
\hookrightarrow \bigoplus_i \bZ[L_i]$. The homomorphism $\Cl(V/G)
\hookrightarrow Q$ associates to a class of a Weil divisor $D$ on $V/G$
the class of $\sum_i (\varphi^{-1}_*D\cdot C_i)[L_i]$ in the quotient
group $Q$.

We conclude this sub-section by observing that the existence of the
symplectic form on a crepant resolution of a quotient singularity
follows by a known result of Beauville, \cite{BeauvilleInv}.

\begin{lemma}\label{sympl-resolution-form}
  Let $G<Sp(V)$ be a finite subgroup preserving a symplectic linear
  2-form on vector space $V$ of dimension $2n$. Suppose that $\varphi:
  X\rightarrow Y=V/G$ is a resolution with exceptional divisors $E_i$,
  where $i=1,\dots, m$. Assume that for every $i$ the image
  $\varphi(E_i)$ is of codimension 2 in $Y=V/G$ and in codimension 2
  the morphism $\varphi$ is a minimal resolution of Du Val
  singularities or $\varphi$ is crepant. Then $\varphi: X\rightarrow Y$
  is a symplectic resolution, that is, $X$ admits a symplectic form.
\end{lemma}

\begin{proof}
  By \cite[Prop.~2.4]{BeauvilleInv} the symplectic form $\omega$ on
  $V$ descends to the smooth locus of $V/G$ and extends to a regular
  closed 2-form $\widetilde{\omega}$ on every resolution $\varphi:
  X\rightarrow Y$. The top exterior power $\widetilde{\omega}^{\wedge
    n}$ does not vanish outside $E_i$'s and because of our assumption
  on $E_i$'s it is non-zero on $E_i$'s as well. Thus
  $\widetilde{\omega}$ is a symplectic form on $X$.
\end{proof}

\subsection{Cyclic group quotients and monomial
  valuations}\label{sect-cyclquot}
In this subsection we discuss a fundamental example of group action
and a blow-up of the resulting quotient.  Let $\epsilon_r=\exp(2\pi
i/r)\in\bC^*$ be the primitive $r$-th root of unity, $r>1$. We
consider the cyclic group $\langle\epsilon_r\rangle\subset\bC^*$. We
assume that the group $\langle\epsilon_r\rangle\iso\bZ_r$ acts
diagonally on the vector space $V$ of dimension $n$ with non-negative
weights $(a_1,\dots, a_n)$, that is $\epsilon_r(x_1,\dots,x_n)=
(\epsilon_r^{a_1}x_1, \dots, \epsilon_r^{a_n}x_n)$, with $0\leq
a_i<r$, and at least two $a_i$'s positive. Moreover, we assume that
the action of $\bZ_r$ is faithful which means that
$(a_1,\dots,a_n,r)=1$.  This action extends to the action of
$\bC^*\supset \langle\epsilon_r \rangle$ with the same weights: for
$t\in\bC^*$ we take $t(x_1,\dots,x_n) = (t^{a_1}x_1,\dots,
t^{a_n}x_n)$.

Let us describe this situation in toric terms. By $M$ let us denote
the lattice with the basis $(u_1,\dots,u_n)$ consisting of characters
of the standard torus $\bT_V$ of $V$ so that $\chi^{u_i}=x_i$. If
$N=M^*$ is the lattice of 1-parameter subgroups of the standard torus
then the $\bC^*$-action in question is defined by $v\in N$ such that
for every $i$ we have $\langle v,u_i\rangle =a_i$.  Equivalently, the
action of $\bC^*$ is associated to a grading $\deg_v$ on
$\bC[V]\subset\bC[M]$ such that $\deg_v(\chi^u)=\langle
v,u\rangle$. The composition of $\deg_v$ with the residue homomorphism
$\bZ\rightarrow\bZ_r$ determines $\bZ_r$-grading $[\deg_v]_r$ on
$\bC[V]$ associated to the action of $\langle\epsilon_r\rangle$. By
$[\deg_v]_r$ we will denote both, the grading on $\bC[V]=\bC[M \cap
\sigma^+]$ and the associated group homomorphism $M\rightarrow\bZ_r$.

Following \cite{ReidMcKay}, \cite{ItoReid},
\cite[Sect.~2]{KaledinMcKay} we state this definition.

\begin{definition}\label{def-monomial-valuation}
  The monomial valuation associated to the cyclic group action
  described above is $\nu_v: \bC(V)\rightarrow \bZ\cup\{\infty\}$ such
  that for $f =\sum_j c_j\chi^{m_j}\ne 0$ it holds
  $\nu_v(f)=\min\{\langle v,m_j\rangle: c_j\ne 0\}$ and moreover
  $\nu_v(f_1/f_2) = \nu_v(f_1)-\nu_v(f_2)$.
\end{definition}

\begin{example}\label{ex-toric-blow-up}
  In the situation introduced above we define a lattice
  $N_{v/r}=N+(v/r)\cdot\bZ$ where the sum is taken in $N_\bQ$. Its
  dual $M_{v/r}\subset M$ consists of characters on which $v/r$
  assumes integral values. In fact, $M_{v/r}$ is the kernel of
  $[\deg_v]_r$. Thus $\bC[M_{v/r} \cap \sigma^+_M] \subset \bC[V]$ is
  the ring of invariants of the action of $\langle \epsilon_r \rangle$
  on $V$, or $Y_{v/r}:=\Spec \bC[M_{v/r} \cap \sigma^+_M]$ is the
  quotient $V/\langle \epsilon_r\rangle$. Equivalently, $Y_{v/r}$ is
  an affine toric variety associated to the cone $\sigma^+_N$ and the
  lattice $N_{v/r}$. We define a toric variety $X_{v/r}$ whose fan is
  defined by taking a ray in $N_{v/r}$ which is generated by $v/r$ and
  sub-dividing the cone $\sigma^+_N$ into a simplicial fan. The
  induced morphism $\varphi: X_{v/r}\rightarrow Y_{v/r}$ is proper and
  birational and its exceptional set is an irreducible divisor
  $E_{v/r}$ which maps to $V^{\langle\epsilon_r\rangle}$ as a weighted
  projective space bundle, with fiber $\bP(a_i: a_i>0)$.

  Let us consider rank $n+1$ lattice $\widehat{M}_{v/r}$ which is a
  sublattice of $M\times\bZ$ generated by $(u_i,a_i)$, for $i=1,\dots,
  n$, and $(0,-r)$. By $\widehat{\sigma}^+$ we denote the cone in
  $M_\bR\times\bR$ spanned by these generators and by $p_2:
  \widehat{M}_{v/r} \rightarrow \bZ$ we denote the projection onto the
  last coordinate. Let us note that the kernel of $p_2$ coincides with
  $M_{v/r}$. If fact, if $p_1: M\times\bZ\rightarrow M$ is the
  projection to the first factor then on $\widehat{M}_{v/r}$ the
  composition $[\deg_v]_r\circ p_1$ coincides with $p_2$ composed with
  the residue homomorphism $\bZ\rightarrow \bZ_r$. By $\psi:
  \bC[\widehat{M}_{v/r} \cap \widehat{\sigma}^+] \rightarrow
  \bC[M\cap\sigma^+_M]$ let us denote the homomorphism of polynomial
  rings induced by the projection $p_1$, so that
  $\psi(\chi^{(u_i,a_i)}) = \chi^{u_i}$ and $\psi(\chi^{(0,-r)})=1$.
\end{example}

\begin{proposition}\label{cyclic-group-blow-up}
  In the above situation the following holds:
  \begin{enumerate}
  \item $\Cl(Y_{v/r})=\bZ_r$ and $\cR(Y_{v/r})=\bC[M\cap\sigma^+_M]$
    with grading $[\deg_v]_r$.
  \item $\Cl(X_{v/r})=\bZ$ and $\cR(X_{v/r})=\bC[\widehat{M}_{v/r}
    \cap \widehat{\sigma}^+]$ with grading $\cR(X_{v/r})=
    \bigoplus_{d\in\bZ} \cR(X_{v/r})_d$ associated to the projection
    $p_2$.
  \item
    $\cR(Y_{v/r})_0=\cR(X_{v/r})_0=\bC[V]^{\langle\epsilon_r\rangle}$
    and if $\cR(X_{v/r})^+= \bigoplus_{d\geq 0} \cR(X_{v/r})_d$ then
    $X_{v/r}=\Proj\cR(X_{v/r})^+    $ and $\cO_{X_{v/r}}(-E_{v/r})=\cO_{\Proj(\cR(X_{v/r}))^+}(r)$
  \item The valuation $\nu_v$ restricted to $\bC(X_{v/r}) =
    \bC(Y_{v/r})\subset\bC(V)$ coincides with $r\cdot\nu_{E}$, where
    $\nu_{E}$ is the divisorial valuation centered at $E_{v/r}$.
  \end{enumerate}
\end{proposition}
\begin{proof}
  The proof uses toric geometry. Let $\widehat{N}_{v/r}$ be a lattice
  dual to $\widehat{M}_{v/r}\subset M\times\bZ$ with the basis
  $v_0,v_1,\dots,v_n$ such that for $\langle v_i,(u_j,a_j)\rangle=1$
  for $1\leq i=j\leq n$, and $\langle v_0,(0,-r)\rangle=1$ and the
  other products are zero. We check that the homomorphism dual to the
  inclusion $M_{v/r}\hookrightarrow\widehat{M}_{v/r}$ is
  $\widehat{N}_{v/r}\rightarrow N_{v/r}$ where $v_i$'s are send to the
  elements of the basis of $N$ and $v_0\mapsto v/r$. Indeed, if
  $u=b_0(0,-r)+\sum_1^n b_i(u_i,a_i)$ is in $M_{v/r}\subset
  \widehat{M}_{v/r}$ then $b_0=\sum_1^n b_i(a_i/r)$ hence the claim
  follows. Now we can use standard arguments in toric geometry, see
  e.g.~\cite{CoxLittleSchenck}.
\end{proof}

\begin{corollary}\label{toric-push-forward}
  Suppose that the situation is as introduced above. Then the
  projection induced homomorphism
  \begin{equation}\label{eq-toric-push-forward-1}
    \psi: \bC[\widehat{M}_{v/r} \cap \widehat{\sigma}^+] =
    \cR(X_{v/d}) \rightarrow \bC[M\cap\sigma^+_M]=\cR(Y_{v/r})
  \end{equation}
  is a homomorphism of graded $\bC[V]^{\langle\epsilon_r\rangle}$
  algebras.  More precisely, for $d\in\bZ$ we have the induced
  injective homomorphism of $d$-th graded pieces $(\psi)_d:
  \cR(X_{v/r})_d \hookrightarrow \cR(Y_{v/r})_{[d]_r}$, as modules
  over $\bC[V]^{\langle\epsilon_r\rangle}=\cR(X_{v/r})_0=
  \cR(Y_{v/r})_0$, and the following holds
  \begin{equation}\label{eq-toric-push-forward-2}
    \begin{array}{ll}
      \psi\left(\cR(X_{v/r})_d\right)& =\left\{f\in\cR(Y_{v/r})_{[d]_r}:
        \nu_v(f)\geq d\right\}
    \end{array}
  \end{equation}
\end{corollary}


\section{The total coordinate ring of symplectic resolutions of a
  quotient}\label{total}

\subsection{The push-forward map}\label{sect-triviality}
We begin this section by discussing a somewhat more general situation,
than what is needed to tackle the problem of our interest.  Let
$\varphi: X\ra Y$ be a projective birational morphism of normal
$\bQ$-factorial varieties which satisfy assumptions formulated at the
beginning of Section \ref{CoxRing}. By $\bQ$-factoriality of $Y$ we know
that the exceptional set of the morphism $\varphi$ is a Weil divisor
with components denoted by $E_i$. The push-forward map of codimension
one cycles $\varphi_*: \Cl(X)\ra\Cl(Y)$ is surjective and its kernel
is generated by the classes of $E_i$'s, c.f.~\ref{class-sequence}.

Moreover, $\varphi$ determines the morphism of the respective total
coordinate rings which we will denote by $\varphi_*$ as well. Namely,
for a reflexive sheaf $\cL=\cO_X(D)$ its reflexive push-forward
$\varphi_*\cL^{\vee\vee}$ is isomorphic to $\cO_Y(\varphi_*D)$. Thus,
pushing down the sections determines the injective homomorphism of
spaces
\begin{equation}\label{triviality-1}
\Gamma(X,\cO_X(D))\ra\Gamma(Y,\cO_Y(\varphi_*(D)))
\end{equation}
associated to the inclusion
\begin{equation}\label{triviality-2}
\{f\in\bC(X)^*:\pdiv(f)+D\geq 0\}\hookrightarrow
\{f\in\bC(Y)^*:\pdiv(f)+\varphi_*(D)\geq 0\}
\end{equation}
This yields the homomorphism of graded rings $\varphi_*:
\cR(X)\ra\cR(Y)$. Note that we use the fact that
$\bC(X)=\bC(Y)$. Well-definedness of the homomorphism $\varphi_*$
follows from the construction in \cite[Sect.~1.4]{CoxRings}, see
\cite{HausenKeicherLaface}. For example, in case when $\Cl(X)$ is
torsion-free, which is the case of our primary interest, c.f.~Lemma
\ref{pic-resolution}, we choose divisors $D_i$ on $X$ whose classes
generate $\Cl(X)$, and use the construction explained in Section
\ref{CoxRing}, see also \cite[1.4.1.1]{CoxRings} to define
$\cR(X)$. Subsequently we use divisors $\varphi_*(D_i)$ and
construction \cite[1.4.2.1]{CoxRings} to define $\cR(Y)$. We summarize
this short general introduction by stating a result from a paper by
Hausen, Keicher and Laface \cite[Prop.~2.2]{HausenKeicherLaface} to
which we refer the reader for details.

\begin{proposition}\label{push-forward-Cox-rings}
  Let $\varphi: X\ra Y$ be a proper birational morphisms of varieties
  which satisfy assumptions stated at the beginning of Section
  \ref{CoxRing}; the exceptional set of $\varphi$ equal to
  $\bigcup_{i=1}^r E_i$, where $E_i$'s are prime divisor. Then there
  exists a canonical surjective homomorphism $\varphi_*:
  \cR(X)\ra\cR(Y)$ which agrees with the homomorphism of gradings
  $\varphi_*:\Cl(X)\ra\Cl(Y)$. The kernel of $\varphi_*$ contains
  elements $1-f_{E_i}$, where $f_{E_i} \in \Gamma(X,\cO(E_i))$ is a
  section defining $E_i$. Moreover, if $\cR(X)$ is finitely generated
  then in fact $\ker\varphi_*=(1-f_{E_i}: i=1,\dots, r)$.
\end{proposition}

The case of a blow-up of a cyclic singularity discussed in Section
\ref{sect-cyclquot} is a particular example of this situation. In
particular, the homomorphism $\psi$ introduced there, is the
push-forward $\varphi_*$ of the respective Cox rings, c.f.~Corollary
\ref{toric-push-forward}.

\begin{example}\label{ex-A1resolution}
  Let us consider the case of a resolution of a surface $A_1$
  singularity. That is $Y=V/\bZ_2$ where $V=\bC^2$ and the non-trivial
  element of $\bZ_2$ acts as $-I$ matrix. The $\varphi:X\ra Y$ is the
  blow-down of a $(-2)$-curve. Clearly the situation is toric and the
  elements of both $\cR(Y)=\bC[V]$ and of $\cR(X)$ are linear
  combinations of monomials which can be visualized as points in a
  lattice of rank 2 or 3, respectively.

  In Figure \ref{caseA1} we present $\varphi_*(\cR(X))_d$ as a
  submodule of $\cR(Y)_{[d]_2}\subset\bC[V]$. The monomials in
  $\bC[V]$ are integral lattice points in the positive quadrant on the
  plane which is indicated by the solid line segments; the dot in the
  left-lower corner is $(0,0)$. These monomials which are in
  $\varphi_*(\cR(X)_d)$ are denoted by $\bullet$ while those which are
  not in $\varphi_*(\cR(X)_d)$ are denoted by $\circ$.  The skewed
  dotted line indicates where the monomial valuation $\nu_{(1,1)}$
  assumes value $d$, c.f.~\ref{eq-toric-push-forward-2}. For $d\leq 1$
  we have $\varphi_*(\cR(X)_d)=\cR(Y)_{[d]_2}$.

\begin{figure}[h]\caption{Resolution of $A_1$ singularity; the image $\varphi_*(\cR(X)_d)$ in
      $\cR(Y)_{[d]_2}$}\label{caseA1}
$$
\begin{array}{cccccccc}
&d\leq 1&&d=2,\ 3&&d=4,\ 5&&d=6,\ 7\\
d\ {\rm even}\\
&
\begin{xy}<10pt,0pt>:
(0,5)*={\circ},(1,5)*={\bullet},(2,5)*={\circ},(3,5)*={\bullet},(4,5)*={\circ},(5,5)*={\bullet},
(0,4)*={\bullet},(1,4)*={\circ},(2,4)*={\bullet},(3,4)*={\circ},(4,4)*={\bullet},(5,4)*={\circ},
(0,3)*={\circ},(1,3)*={\bullet},(2,3)*={\circ},(3,3)*={\bullet},(4,3)*={\circ},(5,3)*={\bullet},
(0,2)*={\bullet},(1,2)*={\circ},(2,2)*={\bullet},(3,2)*={\circ},(4,2)*={\bullet},(5,2)*={\circ},
(0,1)*={\circ},(1,1)*={\bullet},(2,1)*={\circ},(3,1)*={\bullet},(4,1)*={\circ},(5,1)*={\bullet},
(0,0)*={\bullet},(1,0)*={\circ},(2,0)*={\bullet},(3,0)*={\circ},(4,0)*={\bullet},(5,0)*={\circ},
(0,0);(0,6) **@{-}, (0,0);(6,0) **@{-}
\end{xy}&&
\begin{xy}<10pt,0pt>:
(0,5)*={\circ},(1,5)*={\bullet},(2,5)*={\circ},(3,5)*={\bullet},(4,5)*={\circ},(5,5)*={\bullet},
(0,4)*={\bullet},(1,4)*={\circ},(2,4)*={\bullet},(3,4)*={\circ},(4,4)*={\bullet},(5,4)*={\circ},
(0,3)*={\circ},(1,3)*={\bullet},(2,3)*={\circ},(3,3)*={\bullet},(4,3)*={\circ},(5,3)*={\bullet},
(0,2)*={\bullet},(1,2)*={\circ},(2,2)*={\bullet},(3,2)*={\circ},(4,2)*={\bullet},(5,2)*={\circ},
(0,1)*={\circ},(1,1)*={\bullet},(2,1)*={\circ},(3,1)*={\bullet},(4,1)*={\circ},(5,1)*={\bullet},
(0,0)*={\circ},(1,0)*={\circ},(2,0)*={\bullet},(3,0)*={\circ},(4,0)*={\bullet},(5,0)*={\circ},
(0,0);(0,6) **@{-}, (0,0);(6,0) **@{-}, (-1,3);(3,-1) **@{.}
\end{xy}&&
\begin{xy}<10pt,0pt>:
(0,5)*={\circ},(1,5)*={\bullet},(2,5)*={\circ},(3,5)*={\bullet},(4,5)*={\circ},(5,5)*={\bullet},
(0,4)*={\bullet},(1,4)*={\circ},(2,4)*={\bullet},(3,4)*={\circ},(4,4)*={\bullet},(5,4)*={\circ},
(0,3)*={\circ},(1,3)*={\bullet},(2,3)*={\circ},(3,3)*={\bullet},(4,3)*={\circ},(5,3)*={\bullet},
(0,2)*={\circ},(1,2)*={\circ},(2,2)*={\bullet},(3,2)*={\circ},(4,2)*={\bullet},(5,2)*={\circ},
(0,1)*={\circ},(1,1)*={\circ},(2,1)*={\circ},(3,1)*={\bullet},(4,1)*={\circ},(5,1)*={\bullet},
(0,0)*={\circ},(1,0)*={\circ},(2,0)*={\circ},(3,0)*={\circ},(4,0)*={\bullet},(5,0)*={\circ},
(0,0);(0,6) **@{-}, (0,0);(6,0) **@{-}, (-1,5);(5,-1) **@{.}
\end{xy}&&
\begin{xy}<10pt,0pt>:
(0,5)*={\circ},(1,5)*={\bullet},(2,5)*={\circ},(3,5)*={\bullet},(4,5)*={\circ},(5,5)*={\bullet},
(0,4)*={\circ},(1,4)*={\circ},(2,4)*={\bullet},(3,4)*={\circ},(4,4)*={\bullet},(5,4)*={\circ},
(0,3)*={\circ},(1,3)*={\circ},(2,3)*={\circ},(3,3)*={\bullet},(4,3)*={\circ},(5,3)*={\bullet},
(0,2)*={\circ},(1,2)*={\circ},(2,2)*={\circ},(3,2)*={\circ},(4,2)*={\bullet},(5,2)*={\circ},
(0,1)*={\circ},(1,1)*={\circ},(2,1)*={\circ},(3,1)*={\circ},(4,1)*={\circ},(5,1)*={\bullet},
(0,0)*={\circ},(1,0)*={\circ},(2,0)*={\circ},(3,0)*={\circ},(4,0)*={\circ},(5,0)*={\circ},
(0,0);(0,6) **@{-}, (0,0);(6,0) **@{-}, (0,6);(6,0) **@{.}
\end{xy}
\\
d\ {\rm odd}\\
&
\begin{xy}<10pt,0pt>:
(0,5)*={\bullet},(1,5)*={\circ},(2,5)*={\bullet},(3,5)*={\circ},(4,5)*={\bullet},(5,5)*={\circ},
(0,4)*={\circ},(1,4)*={\bullet},(2,4)*={\circ},(3,4)*={\bullet},(4,4)*={\circ},(5,4)*={\bullet},
(0,3)*={\bullet},(1,3)*={\circ},(2,3)*={\bullet},(3,3)*={\circ},(4,3)*={\bullet},(5,3)*={\circ},
(0,2)*={\circ},(1,2)*={\bullet},(2,2)*={\circ},(3,2)*={\bullet},(4,2)*={\circ},(5,2)*={\bullet},
(0,1)*={\bullet},(1,1)*={\circ},(2,1)*={\bullet},(3,1)*={\circ},(4,1)*={\bullet},(5,1)*={\circ},
(0,0)*={\circ},(1,0)*={\bullet},(2,0)*={\circ},(3,0)*={\bullet},(4,0)*={\circ},(5,0)*={\bullet},
(0,0);(0,6) **@{-}, (0,0);(6,0) **@{-}
\end{xy}
&&
\begin{xy}<10pt,0pt>:
(0,5)*={\bullet},(1,5)*={\circ},(2,5)*={\bullet},(3,5)*={\circ},(4,5)*={\bullet},(5,5)*={\circ},
(0,4)*={\circ},(1,4)*={\bullet},(2,4)*={\circ},(3,4)*={\bullet},(4,4)*={\circ},(5,4)*={\bullet},
(0,3)*={\bullet},(1,3)*={\circ},(2,3)*={\bullet},(3,3)*={\circ},(4,3)*={\bullet},(5,3)*={\circ},
(0,2)*={\circ},(1,2)*={\bullet},(2,2)*={\circ},(3,2)*={\bullet},(4,2)*={\circ},(5,2)*={\bullet},
(0,1)*={\circ},(1,1)*={\circ},(2,1)*={\bullet},(3,1)*={\circ},(4,1)*={\bullet},(5,1)*={\circ},
(0,0)*={\circ},(1,0)*={\circ},(2,0)*={\circ},(3,0)*={\bullet},(4,0)*={\circ},(5,0)*={\bullet},
(0,0);(0,6) **@{-}, (0,0);(6,0) **@{-}, (-1,4);(4,-1) **@{.}
\end{xy}
&&
\begin{xy}<10pt,0pt>:
(0,5)*={\bullet},(1,5)*={\circ},(2,5)*={\bullet},(3,5)*={\circ},(4,5)*={\bullet},(5,5)*={\circ},
(0,4)*={\circ},(1,4)*={\bullet},(2,4)*={\circ},(3,4)*={\bullet},(4,4)*={\circ},(5,4)*={\bullet},
(0,3)*={\circ},(1,3)*={\circ},(2,3)*={\bullet},(3,3)*={\circ},(4,3)*={\bullet},(5,3)*={\circ},
(0,2)*={\circ},(1,2)*={\circ},(2,2)*={\circ},(3,2)*={\bullet},(4,2)*={\circ},(5,2)*={\bullet},
(0,1)*={\circ},(1,1)*={\circ},(2,1)*={\circ},(3,1)*={\circ},(4,1)*={\bullet},(5,1)*={\circ},
(0,0)*={\circ},(1,0)*={\circ},(2,0)*={\circ},(3,0)*={\circ},(4,0)*={\circ},(5,0)*={\bullet},
(0,0);(0,6) **@{-}, (0,0);(6,0) **@{-},  (-1,6);(6,-1) **@{.}
\end{xy}
&&
\begin{xy}<10pt,0pt>:
(0,5)*={\circ},(1,5)*={\circ},(2,5)*={\bullet},(3,5)*={\circ},(4,5)*={\bullet},(5,5)*={\circ},
(0,4)*={\circ},(1,4)*={\circ},(2,4)*={\circ},(3,4)*={\bullet},(4,4)*={\circ},(5,4)*={\bullet},
(0,3)*={\circ},(1,3)*={\circ},(2,3)*={\circ},(3,3)*={\circ},(4,3)*={\bullet},(5,3)*={\circ},
(0,2)*={\circ},(1,2)*={\circ},(2,2)*={\circ},(3,2)*={\circ},(4,2)*={\circ},(5,2)*={\bullet},
(0,1)*={\circ},(1,1)*={\circ},(2,1)*={\circ},(3,1)*={\circ},(4,1)*={\circ},(5,1)*={\circ},
(0,0)*={\circ},(1,0)*={\circ},(2,0)*={\circ},(3,0)*={\circ},(4,0)*={\circ},(5,0)*={\circ},
(0,0);(0,6) **@{-}, (0,0);(6,0) **@{-},  (1,6);(6,1) **@{.}
\end{xy}
\end{array}
$$
\end{figure}
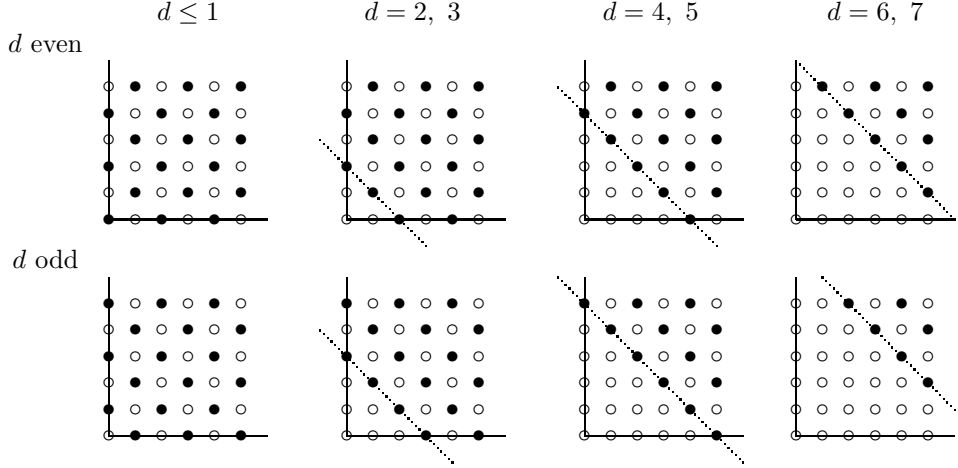
\end{example}

\subsection{Cox ring of a resolution of  a quotient
  singularity}\label{sect-Cox_ring_quotsing}

From now on we consider the case which is of our primary interest. Let
$G\subset GL(V)$ be a finite group without quasi-reflections with
$Y=V/G=\Spec\bC[V]^G$ the quotient. Suppose that $\varphi: X\ra Y$ is
a resolution of singularities. Now, because of \ref{pic-resolution},
$\Cl(X)=\Pic(X)$ is free abelian of rank, say, $m$.  Then the group
$\Hom(\Cl(X),\bC^*)$ is an algebraic torus $\bT=\bT_{\Cl(X)}\iso
(\bC^*)^m$ with the coordinate ring $\bC[\Cl(X)]$. The grading of
$\cR(X)$ in $\Cl(X)$ is associated to the action of $\bT$ on $\cR(X)$
with the ring of invariants equal to $\cR(X)_0=\bC[V]^G$.

If $\cR(X)$ is finitely generated $\bC$-algebra then we can present the
emerging objects in a single diagram
\begin{equation}\label{quotients-diagram}
   \xymatrix
   {\Spec\cR(X)\ \ar[dr]_{\bT}\ar@{-->}[r]&
     X\ar[d]&\ V\ar[dl]_{G}\ar[d]^{[G,G]}\\
     &V/G&\ V/[G,G]\ar[l]^{Ab(G)\ }\ar
     [ull]
   }
\end{equation}
where $X\ra V/G$ is the resolution of singularities. Moreover, $V\ra
V/[G,G]\ra V/G$ and $\Spec\cR(X)\ra V/G$ are the (categorical)
quotients with respect to appropriate group actions (spectra of rings
of invariants), and the rational map $\Spec\cR(X)\dashrightarrow X$ is a GIT
quotient. The morphism of affine schemes $V/[G,G]=\Spec\cR(V/G) \ra
\Spec\cR(X)$ is the map of varieties associated to $\varphi_*:
\cR(X)\ra\cR(V/G)$ introduced above in \ref{push-forward-Cox-rings}.

The action of the torus $\bT=\bT_{\Cl(X)}$ on $\Spec\cR(X)$ is
associated to the multiplication homomorphism
\begin{equation}\label{triviality-3}
  \cR(X)\ra\cR(X)\otimes\bC[\Cl(X)]
\end{equation}
sending $f\in\Gamma(X,\cO_X(D))$ to $f\cdot \chi^{[D]}$, with class
$[D]\in \Cl(X)$ defining the character $\chi^{[D]}$ of the torus in
question. Here we make an identification $\cR(X)\otimes\bC[\Cl(X)] =
\cR(X)[\Cl(X)]$ of the tensor product with Laurent polynomials with
coefficients in $\cR(X)$.

We define a map
\begin{equation}\label{triviality}
  \Theta:\cR(X)\ra\bC[V]^{[G,G]}\otimes\bC[\Cl(X)]
\end{equation}
which to $f\in\Gamma(X,\cO_X(D))$ associates $f\cdot \chi^{[D]}\in
\Gamma(Y,\cO_Y(\varphi_*(D)))\otimes \chi^{[D]}$, where $\chi^{[D]}$
denotes the character of $\bT$. That is, $\Theta$ is a composition of
the pushing down \ref{triviality-1} and multiplication
\ref{triviality-3}.

\begin{proposition}\label{emmbedding}
  The map $\Theta$ defined above is injective.
\end{proposition}
\begin{proof}
  If $\Theta(f_1)=\Theta(f_2)$ then they are both in the same space
  $\Gamma(X,\cO_X(D))$. However the map
  $\Gamma(X,\cO_X(D))\ra\Gamma(Y,\cO_Y(\varphi_*(D)))$ is injective
  hence the claim.
\end{proof}

Now we know that the total coordinate ring of $X$ can be realized as a
subring of the known ring $\bC[V]^{[G,G]} \otimes \bC[\Cl(X)] =
\cR(V/G)[\Cl(X)]$; the problem now is to construct generators of this
subring.

If $G\subset Sp(V)$ is a symplectic group and $\varphi: X\ra V/G$ a
symplectic resolution then we are in situation of
\ref{class-sequence-1} and the description of elements of $\cR(X)$ can
be made even more transparent. Recall that $[L_i]$, with $i=1,\dots,m$
is a $\bQ$-basis of $\NU(X)$, dual to the classes of $C_i$'s,
components of fibers of $\varphi_{|E_i}$, \ref{class-sequence-1}; then
we have embedding of lattices $\Cl(X) \hookrightarrow \bigoplus_i
\bZ[L_i]$. This yields an embedding $\bC[\Cl(X)] \hookrightarrow
\bC[t_1^{\pm 1},\dots,t_m^{\pm 1}]$ where $t_i$'s are variables
associated to $[L_i]$'s, that is $t_i=\chi^{[L_i]}$. By
\begin{equation}\label{triviality-extended}
\overline\Theta:\cR(X)\ra\cR(V/G)[t_1^{\pm 1},\dots, t_m^{\pm 1}]
\end{equation}
we denote the composition of the homomorphism $\Theta$ with the
extension of coefficients $\bC[\Cl(X)] \hookrightarrow \bC[t_1^{\pm
  1},\dots,t_m^{\pm 1}]$. Let us note the following consequence of the
construction of $\overline{\Theta}$.

\begin{lemma}\label{composition-Theta-evaluation}
  The composition of $\overline{\Theta}:\cR(X)\ra\cR(V/G)[t_1^{\pm
    1},\dots, t_m^{\pm 1}]$ with the evaluation homomorphism $ev_1:
  \cR(V/G)[t_1^{\pm 1},\dots,t_m^{\pm 1}]\ra \cR(V/G)$ such that
  $ev_1(t_i)= 1$ for every $i=1,\dots, m$, is equal to the
  push-forward homomorphism $\varphi_*:\cR(X)\ra\cR(V/G)$.
\end{lemma}

\begin{corollary}\label{emmbedding-1}
  Assume that we are in the situation discussed above.  Let
  $f_D\in\cR(V/G)=\bC[V]^{[G,G]}$ be a non-zero element associated to
  an effective Weil divisor $D$ on $V/G$. Let
  $\overline{D}=\varphi^{-1}_* D$ be its strict transform in $X$. If
  $f_{\overline{D}}\in\cR(X)$ is associated to $\overline{D}$ in the
  total coordinate ring of $X$ then
  $$\overline{\Theta}(f_{\overline{D}})=
  f_D\cdot\prod_i t_i^{(\overline{D}\cdot C_i)}$$
\end{corollary}
\begin{proof}
  The $\cR(V/G)$-coefficient of $f_{\overline{D}}$ is $f_D$ and it
  remains to verify the degree of $f_{\overline{D}}$ with respect to
  $\Cl(X)$ which is provided by the homomorphism $\Cl(X)
  \hookrightarrow \bigoplus_i\bZ[L_i]$ in \ref{class-sequence-1}.
\end{proof}


\subsection{From the group $G$ to a torus $\bT$}\label{section-torus}

From this point on by $G$ we denote the group introduced in Section
\ref{the-group}.  The ring of invariants of $[G,G]=\langle -I_V\rangle$
is generated by quadratic forms in $\bC[x_1,x_2,x_3,x_4]$. We note
that the linear space of forms is $S^2V^*$. The following observation
can be verified easily.

\begin{lemma}\label{eigenvectors}
  The action of $Ab(G)=\bZ_2^4$ on $S^2V^*$ yields a decomposition of
  $S^2V^*$ into the sum of 1-dimensional eigenspaces generated by the
  functions $\F_{ij}$ given in the following table. The action of the
  class of $T_i$ in $Ab(G)$ on the function $\F_{rs}$ is by
  multiplication by $\pm 1$, as indicated in the following table:
  \begin{equation}\label{eq-eigenvectors}
    \begin{array}{llccccc}
      {\rm }&{\rm function} &T_0&T_1&T_2&T_3&T_4\\
      \F_{01}=&  -2(x_1x_4+x_2x_3)         &-&-&+&+&+\\
      \F_{02}=&  2\sqrt{-1}(-x_1x_4+x_2x_3)        &-&+&-&+&+\\
      \F_{03}=&  2\sqrt{-1}(x_1x_2+x_3x_4)         &-&+&+&-&+\\
      \F_{04}=&  2(-x_1x_2+x_3x_4)         &-&+&+&+&-\\
      \F_{12}=&  2(x_1x_3-x_2x_4)         &+&-&-&+&+\\
      \F_{13}=&  -x_1^2-x_2^2+x_3^2+x_4^2  &+&-&+&-&+\\
      \F_{14}=& \sqrt{-1}( x_1^2+x_2^2+x_3^2+x_4^2)&+&-&+&+&-\\
      \F_{23}=& \sqrt{-1}(-x_1^2+x_2^2-x_3^2+x_4^2)&+&+&-&-&+\\
      \F_{24}=&  x_1^2-x_2^2-x_3^2+x_4^2 &+&+&-&+&-\\
      \F_{34}=&  2(x_1x_3+x_2x_4)         &+&+&+&-&-\\
    \end{array}
  \end{equation}
\end{lemma}

The labeling of functions $\F_{rs}$ indicates an isomorphism between
$S^2V^*$ and $\bigwedge^2W^*$ where $W$ is a 5-dimensional space with
coordinates $t_0,\dots, t_4$; under this isomorphism the function
$\F_{rs}$ is associated to the 2-form $t_r\wedge t_s$.

  In fact, the representation of $Sp(V)$ on $\bigwedge^2V$
  splits into $W\bigoplus\bC$ where $\bC$ stands for the trivial
  representation and the action of $Sp(4,\bC)$ on $W$ is associated to
  the double cover $Sp(4,\bC)\ra SO(5,\bC)$ whose kernel is $-I$. Then
  we have the natural isomorphism of $Sp(4,\bC)$ representations
  $\bigwedge^2W\iso S^2V$, \cite[16.2]{FultonHarris}. The
  coordinates $\F_{rs}$ diagonalize the induced action of $T_i$'s.

Let $\bT_W$ be the standard torus of $W$ with
$\Hom(\bT_W,\bC^*)\iso\bigoplus_{i=0}^4 \bZ e_i$ and characters
$t_0=\chi^{e_0},\dots, t_4=\chi^{e_4}$. Let
$\Lambda\subset\Hom(\bT_W,\bC^*)$ be the index 2 sublattice of
$\bigoplus_{i=0}^4 \bZ e_i$ consisting of characters invariant by
multiplication by $-I_W$; that is $\Lambda=\{\sum a_ie_i: a_i\in\bZ,\
2|\sum a_i\}$. We have a surjective morphism $\bT_W\ra\bT_\Lambda$
with kernel $\langle -I_W\rangle$ which is associated to the inclusion
of lattices of characters. Since $-I_W$ acts trivially on
$\bigwedge^2W^*$ the action of $\bT_W$ on $\bigwedge^2W^*$ descends to
the action of $\bT_\Lambda$.

Let $\widetilde{T}_i: W^*\ra W^*$ be a homomorphism defined as follows
$\widetilde{T}_i(t_i)=-t_i$, $\widetilde{T}_i(t_j)=t_j$, for $j\ne
i$. We have injection $\bigoplus_{i=0}^4\bZ_2\widetilde{T}_i
\hookrightarrow \bT_W$ and thus a morphism $\bigoplus_{i=0}^4
\bZ_2\widetilde{T}_i \ra \bT_\Lambda$ with kernel $\langle
-I_W\rangle$. We summarize this discussion in the following.

\begin{lemma}\label{equivariant}
  The homomorphism of groups $Ab(G)=G/[G,G]\ra \bT_\Lambda$ which maps
  the class of $\pm T_i$ in $Ab(G)$ to the class of
  $\pm\widetilde{T}_i$ in $\bT_\Lambda$ makes the isomorphism $S^2V^*
  \iso \bigwedge^2W^*$ equivariant with respect to the action of
  $Ab(G)$.
\end{lemma}

We note that the above homomorphism $Ab(G)\ra\bT_\Lambda$ can be
described in terms of characters of these groups.  Let
$\bigoplus_{i=0}^4 \bZ e_i\ra\bigoplus_{i=0}^4 \bZ_2 e_i$ be the
reduction modulo 2. The latter group can be interpreted as the group
of characters of $\bigoplus_{i=0}^4 \bZ_2 \widetilde{T}_i$.  The
morphism $\bigoplus_{i=0}^4 \bZ_2 \widetilde{T}_i\ra Ab(G)$ which maps
$\widetilde{T}_i$ to $[\pm T_i]$ implies inclusion $Ab(G)^\vee
\hookrightarrow (\bigoplus_{i=0}^4 \bZ_2 \widetilde{T}_i)^\vee =
\bigoplus_{i=0}^4 \bZ_2 e_i$ and because of the inclusion
$\Lambda\hookrightarrow \bigoplus_{i=0}^4 \bZ e_i$ we get a surjective
homomorphism of groups of characters $\Lambda\ra Ab(G)^\vee$.

\begin{definition}\label{def-CR-resolution}
  Let $\cR\subset\bC[V]\otimes\bC[\bT_W]=\bC[x_1,\dots,x_4,t_0^{\pm
    1},\dots, t_4^{\pm 1}]$ be the subring generated by the following
  functions:
  \begin{itemize}
  \item $\F_{ij}\cdot t_it_j$, where $0\leq i < j\leq 4$,
  \item $t_i^{-2}$, where $i=0,\dots,4$
  \end{itemize}
\end{definition}

The torus $\bT_W$ acts naturally on $\bC[V]\otimes\bC[\bT_W]$ by
multiplication of the right factor and the inclusion
$\cR\subset\bC[V]\otimes\bC[\bT_W]$ is $\bT_W$ equivariant.  We see
that $-I_W$ acts on $\cR$ trivially so the action of $\bT_W$ on $\cR$
descends to the action of $\bT_\Lambda$. Note that we have a
surjective homomorphism $\cR\ra\bC[V]^{\langle
  -I_V\rangle}\subset\bC[V]$ obtained by setting $t_i\mapsto 1$.

\begin{proposition}\label{map-of-invariants}
  The induced homomorphism $\cR^{\bT_W}\ra\bC[V]^{[G,G]}\subset\bC[V]$
  is an injection onto the ring of invariants $\bC[V]^G$. Therefore
  $\cR^{\bT_W}\iso \bC[V]^G$.
\end{proposition}
\begin{proof}
  As noted in \ref{equivariant} we have an injection
  $Ab(G)\hookrightarrow \bT_\Lambda$ and we claim that the morphism
  $\cR \ra \bC[V]^{[G,G]}$ is $Ab(G)$-equivariant. Indeed, the action
  of $Ab(G)\hookrightarrow\bT_W$ on $\F_{ij}t_it_j$ agrees with that
  of $G$ on $\F_{ij}$, while on $t_i^2$ the group
  $Ab(G)\hookrightarrow\bT_W$ acts trivially. Therefore, in
  particular, we have $\cR^{\bT_W}\ra\bC[V]^G$. The injectivity of
  this homomorphism is clear since $(t_i-1, i=0,\dots,4) \cap
  \bC[\bT_W]^{\bT_W}=(0)$. It remains to prove surjectivity. To this
  end, we note that a monomial $\prod_{i,j}\F_{ij}^{a_{ij}}\in
  \bC[V]^{\langle -I\rangle}$ is $G$-invariant if, for every
  $k=0,\dots,4$, the sum $s_k=\sum_{k\in\{i,j\}}a_{ij}$ is divisible
  by 2. But then the monomial in question is the image of the $\bT_W$
  invariant monomial $\prod_{i,j}(\F_{ij}\cdot t_it_j)^{a_{ij}}
  \prod_k(t_k^{-2})^{s_k/2}\in \cR$.
\end{proof}


\subsection{Generators of ideals}\label{section-generators}
We present the ring $\cR$ as the quotient ring of the graded
polynomial ring $\bC[w_{ij}, u_k: k=0,\dots,4, 0\leq i<j\leq 4]$ with
the grading in $\Hom(\bT_W,\bC^*)\iso \bigoplus_{m=0}^4 \bZ\cdot e_m$ given
by the formula $\deg w_{ij}=e_i+e_j$, $\deg u_k=-2e_k$.
\begin{proposition}\label{ideal-gens}
  The homomorphism $\bC[w_{ij}, u_k: k=0,\dots,4, 0\leq i<j\leq 4]\ra
  \cR$ which sends $w_{ij}$ to $\F_{ij}t_it_j$ and $u_k$ to $t_k^{-2}$
  is surjective and preserves grading. Its kernel, denoted by
  $\cI$, is generated by the following homogeneous polynomials
$$\begin{array}{ll}
w_{14}w_{23}+w_{13}w_{24}-w_{12}w_{34}&w_{04}w_{23}-w_{03}w_{24}-w_{02}w_{34}\\
w_{04}w_{13}+w_{03}w_{14}-w_{01}w_{34}&w_{04}w_{12}-w_{02}w_{14}-w_{01}w_{24}\\
w_{03}w_{12}+w_{02}w_{13}-w_{01}w_{23}\\
w_{02}w_{12}u_2-w_{03}w_{13}u_3+w_{04}w_{14}u_4&
w_{01}w_{14}u_1-w_{02}w_{24}u_2+w_{03}w_{34}u_3\\
w_{01}w_{13}u_1+w_{02}w_{23}u_2+w_{04}w_{34}u_4&
w_{01}w_{12}u_1+w_{03}w_{23}u_3+w_{04}w_{24}u_4\\
w_{03}w_{04}u_0-w_{13}w_{14}u_1+w_{23}w_{24}u_2&
w_{02}w_{04}u_0+w_{12}w_{14}u_1+w_{23}w_{34}u_3\\
w_{01}w_{04}u_0+w_{12}w_{24}u_2+w_{13}w_{34}u_3&
w_{02}w_{03}u_0-w_{12}w_{13}u_1-w_{24}w_{34}u_4\\
w_{01}w_{03}u_0+w_{12}w_{23}u_2+w_{14}w_{34}u_4&
w_{01}w_{02}u_0+w_{13}w_{23}u_3-w_{14}w_{24}u_4\\
w_{02}^2u_0+w_{12}^2u_1+w_{23}^2u_3+w_{24}^2u_4&
w_{03}^2u_0+w_{13}^2u_1+w_{23}^2u_2+w_{34}^2u_4\\
w_{01}^2u_1+w_{02}^2u_2+w_{03}^2u_3+w_{04}^2u_4&
w_{04}^2u_0+w_{14}^2u_1+w_{24}^2u_2+w_{34}^2u_3\\
w_{01}^2u_0+w_{12}^2u_2+w_{13}^2u_3+w_{14}^2u_4
\end{array}$$
\end{proposition}
\begin{proof}
  The first part is clear. The second part, that is the generators of
  the kernel $\cI$, are obtained by computer calculation.
\end{proof}

The next observation follows from \ref{Cox-P^2_4}.
\begin{corollary}\label{emmbeddP^2_4}
  The ideal $\cI_0=\cI+(u_0,\dots, u_4)$ descends to an ideal in
  $\bC[w_{ij}]=\bC[w_{ij},u_k]/(u_0,\dots,u_4)$ which is an ideal of
  the affine cone over the Grassmann variety $Gr(2,W)$ embedded via
  Pl\"ucker embedding in $\bP(\bigwedge^2 W^*)$ with the associated
  grading coming from the action of $\bT_W$. In particular, we can
  identify $\Cl(\bP^2_4)$ with $\Lambda$ and we have a $\bT_\Lambda$
  equivariant embedding $\Spec\cR(\bP^2_4)\hookrightarrow\Spec\cR$.
\end{corollary}

\subsection{Monomial valuations and the total coordinate
  ring}\label{sect-tot-coord}
In Section \ref{section_proofs} we will produce a GIT quotient $X$ of
$\Spec\cR$ and prove that it is smooth so that the resulting morphism
$\varphi: X\ra V/G$ is a resolution of singularities, see Theorem
\ref{thmGITquotient}. The aim of the present subsection is to relate
$\cR$ to $\cR(X)$, the total coordinate ring of the resolution $X$.

To simplify the notation we set $\cP:=\bC[V]^{\langle
  -I\rangle}=\cR(V/G)$. The ring $\cR$ is constructed as a subring of
the Laurent polynomial ring $\cP[t_0^{\pm 1},\dots,t_4^{\pm 1}]$ with
grading in a lattice $\Lambda\subset\bZ^5$ inherited from the ambient
polynomial ring, see \ref{def-CR-resolution}. In this construction
each variable $t_i$ is associated to the action of the symplectic
reflection $T_i$, \ref{equivariant}.

On the other hand the morphism $\overline{\Theta}:
\cR(X)\ra\cP[t_0^{\pm 1},\dots,t_4^{\pm 1}]$, defined in
\ref{triviality-extended}, is an embedding, see \ref{emmbedding}. The
grading of $\cR(X)$ is in $\Cl(X)$ which embeds via
\ref{class-sequence-1} into $\bZ^5$ so that $\Cl(X)=\Lambda$. Now,
however, $t_i$'s are variables associated to exceptional divisors
$E_i$ of the symplectic resolution $\varphi$, which by McKay
correspondence \ref{KaledinMcKay} are in relation with the conjugacy
classes of $T_i$'s. The composition of $\overline{\Theta}$ with
evaluation $ev_1: t_i\mapsto 1$ is $\varphi_*: \cR(X)\ra\cP$, see
\ref{composition-Theta-evaluation}. The restriction of the evaluation
$ev_1$ to $\cR$ will be called $\Phi:\cR\ra\cP$

The definition of both $\cR$ and $\overline{\Theta}(\cR(X))$ depends
on the meaning of $t_i$'s; the link is McKay correspondence. Thus in
order to relate these objects, following \cite{ReidMcKay} and
\cite{KaledinMcKay}, we will use monomial valuations. Namely, by
$\nu_i: (\cP)=\bC(V)^{\langle -I\rangle}\ra \bZ\cup\{\infty\}$ we
denote the monomial valuation on the field of fractions of $\cP$
associated to the action of $T_i$. By \cite{KaledinMcKay} we know that
$(\nu_i)_{|\bC(V)^G}=2\nu_{E_i}$, c.f.~\ref{cyclic-group-blow-up}.

Because of \ref{Coxringdecomposition} we have the following
decomposition into a sum of $\bC[V]^G$-modules of eigenfunctions of
the action of $Ab(G)$:
\begin{equation}\label{decomposition-2}
  \cP=\bigoplus_{\mu\in G^\vee}\bC[V]^G_\mu=
\bigoplus_{\mu\in G^\vee}\cP_{(\bar\mu(T_0),\dots,\bar\mu(T_4))}
\end{equation}
where $\bar\mu(T_i)=0$ if $\mu(T_i)=1$ and $\bar\mu(T_i)=1$ if
$\mu(T_i)=-1$. Note that this makes grading of $\cP$ in $\bZ_2^5$.
The grading on $\cP$ agrees with the $\bZ^5$ grading on $\cR$ and
$\cR(X)$ as well as with the valuations $\nu_i$.

\begin{lemma}\label{ZandZ2-grading}
  For every $f\in\cP$ if a monomial $f\cdot t_0^{d_0}\cdots t_4^{d_4}$
  is in either $\cR$ or $\cR(X)$ then
  $f\in\cP_{([d_0]_2,\dots,[d_4]_2)}$. If $f\in\cP_{(d_0,\dots,d_4)}$
  then for every $i=0,\dots, 4$ the valuation $\nu_i(f)$ has the same
  parity as $d_i$.  In particular, $\nu_r(\phi_{ij})=1$ if
  $r\in\{i,j\}$ and $\nu_r(\phi_{ij})=0$ if $r\not\in\{i,j\}$.
\end{lemma}

\begin{proof}
  For $f\cdot t_0^{d_0}\cdots t_4^{d_4}\in\cR$ it is enough to check
  the statement for generators of $\cR$, see \ref{def-CR-resolution}.
  If $f\cdot t_0^{d_0}\cdots t_4^{d_4}\in\cR(X)$ then the statement
  follows from the definition of $\Theta$, see \ref{triviality}. The
  last part follows directly from the definition
  \ref{def-monomial-valuation}.
\end{proof}

\begin{example}\label{pull-back-phi}
  Let $\overline{D}_{ij}$ be a divisor on $X$ which is strict
  transform via $\varphi^{-1}$ of the Weil divisor $D_{ij}$ on $V/G$
  associated to $\phi_{ij} \in \cR(V/G) = \cP$. Then the principal
  divisor on $X$ of the function $\phi^2_{ij} \in \bC[V]^G \subset
  \bC(X)$ satisfies the following equality: $\pdiv_X(\phi_{ij}^2) =
  2\overline{D}_{ij} + E_i + E_j$.  Thus, if $C_m$ is a general fiber
  of $\varphi_{|E_m}$ then $\overline{D}_{ij}\cdot C_m=1$ if
  $m\in\{i,j\}$ and $\overline{D}_{ij}\cdot C_m=0$ if
  $m\not\in\{i,j\}$.
\end{example}

More generally we have the following.

\begin{lemma}\label{valuation-strict-transform}
  Let $\overline{D}$ be a divisor in $X$ which is strict transform
  via $\varphi^{-1}$ of an effective Weil divisor $D$ on $V/G$. If
  $f_D\in\cR(V/G) = \cP$ is the element associated with $D$ then
  $\nu_i(f_D) = \overline{D}\cdot C_i$. Moreover, if
  $f_{\overline{D}}\in\cR(X)$ is the element associated with
  $\overline{D}$ then
  $$\overline{\Theta}(f_{\overline{D}})=f_D\cdot
  t_0^{\nu_0(f_D)}\cdots t_4^{\nu_4(f_D)}$$
\end{lemma}

\begin{proof}
  Note that $f_D^2\in\bC[V]^G\subset\bC(X)$ and we have
  $$\pdiv_X(f_D^2)=2\overline{D}+\nu_{E_0}(f_D^2)E_0+\cdots+
  \nu_{E_4}(f_D^2)E_4$$ Since $\pdiv_X(f_D^2)\cdot C_i=0$ and
  $(\nu_i)_{|\bC(X)}=2\nu_{E_i}$ we get
  $$2\overline{D}\cdot C_i= -\left(\nu_0(f_D)E_0+\cdots+
    \nu_4(f_D)E_4\right)\cdot C_i$$ and the first claim follows
  because $D_i\cdot C_i=-2$ and $D_j\cdot C_i=0$ if $j\ne i$. The
  second statement follows from \ref{emmbedding-1}.
\end{proof}

\begin{corollary}\label{inclusion}
  In the notation introduced above the following holds:
  $\overline{\Theta}(f_{E_i})= t_i^{-2}$ and
  $\overline{\Theta}(f_{\overline{D}_{ij}})= \phi_{ij}t_it_j$.
  Therefore $\cR\subseteq\overline{\Theta}(\cR(X))$.
\end{corollary}

The following result has been anticipated in the cyclic quotient case,
see \ref{eq-toric-push-forward-2} and \ref{ex-A1resolution}.

\begin{proposition}\label{grading-valuations-C(R)}
  The image via $\varphi_*$ of the graded pieces of $\cR(X)$ is
  determined by valuations $\nu_i$ in the following way:
\begin{equation}\label{eq-grading-valuations-C(R)}
  \varphi_*(\cR(X)_{(d_0,\dots,d_4)})=
  \left\{f\in\cP_{([d_0]_2,\dots,[d_4]_2)}:\ \forall_i
    \ \nu_i(f) \geq d_i\right\}
\end{equation}
\end{proposition}
\begin{proof}
  We use the notation from Lemma \ref{valuation-strict-transform}. If
  $f=f_D$ and $d_i\leq\nu_i(f)$ then the numbers $a_i=(\nu_i(f)-d_i)/2$
  are non-negative integers because of \ref{ZandZ2-grading} and
  $$f_M:=f_{\overline{D}}\cdot f_{E_0}^{a_0}\cdots f_{E_4}^{a_4}\in
  \cR(X)_{(d_0,\dots,d_4)}$$ is the element which is mapped to $f_D$ via
  $\varphi_*$.

  On the other hand, given an effective divisor $M$ on $X$ we can
  write it as $M=\overline{D}+\sum a_iE_i$ where $\overline{D}$ is the
  strict transform of $D:=\varphi_*(M)$ and $a_i$ are non-negative
  integers. If $f_M=f_{\overline{D}}\cdot f_{E_0}^{a_0}\cdots
  f_{E_4}^{a_4}$ is in $ \cR(X)_{(d_0,\dots,d_4)}$ then by the same
  arguments $\nu_i(\varphi_*(f_M))=\nu_i(f_D)=d_i+2a_i$.
\end{proof}


\section{GIT quotients of $\Spec \cR$}\label{section_proofs}

We study linearizations and corresponding GIT quotients of $\Spec \cR$. For a chosen one we prove its smoothness and this way we obtain an explicit description of a resolution of~$V/G$. In section~\ref{geom} we will use these results to show how to modify this resolution to obtain all other ones.

\subsection{Linearization, stability and isotropy}\label{section_irrelevant_ideal}

To construct a GIT quotient of $\Spec \cR$ we need to choose a suitable linearization of the trivial line bundle. It will be represented by a character $\chi^u \colon \bT_\Lambda \ra \bC^*$ of the 5-dimensional torus.
We investigate the sets of stable and semi-stable points of $\Spec \cR$ with respect to~$\chi$. In this section we explain how to check whether~$\chi^u$ and these sets have properties needed to have a good description of the quotient, i.e. satisfy condition~\ref{condition_good_lin}. Note that we do not compute the irrelevant ideal, i.e. the ideal of the closed set of unstable points, explicitly -- we prefer to deal with the set of semi-stable points using a description based on toric geometry, as explained below. The idea, similar as for the 2-dimensional quotients in~\cite[Sect.~4]{DontenSurfaces}, is to look at the embedding $$\Spec \cR \hookrightarrow \Spec \bC[w_{ij}, u_k: k=0,\dots,4, 0\leq i<j\leq 4] \simeq \bC^{15}$$ such that $\bT_\Lambda$ is a subtorus of the big torus $(\bC^*)^{15}$ of the affine space, as described in~\ref{section-torus}. Then the set of semi-stable points can be presented as the intersection of~$\Spec \cR$ with certain orbits of the big torus, see~\ref{lemma_stab_toric} and Section~\ref{section_stable_points}.

We start from a few observations in a slightly more general setting. Let $Z$ be an affine subvariety of $A \simeq \bC^r$, invariant under a (diagonal) action of a subtorus~$\bT$ of $\bT_A  \simeq (\bC^*)^r$. By $M_{\bT}$ and $\hM$ we denote monomial lattices of $\bT$ and $\bT_A$ respectively, and by $\sigma^+$ and $\hs^+$ their positive orthants. Then both $\bC[Z]$ and $\bC[A]$ have a grading by $M_{\bT}$ associated with the action of $\bT$. To analyze semi-stability we use the notion of orbit cones, see~\cite[Def.~2.1]{BerchtoldHausenGIT}.

\begin{definition}
The orbit cone $\omega_{\bT}(z) \subset M_{\bT} \otimes \bR$ of $z \in Z$ is a convex (polyhedral) cone generated by $$\{u \in M_{\bT} \colon \exists f \in \bC[Z]_u\: f(z) \neq 0\}$$
where $\bC[Z]_u$ denotes the graded piece in degree~$u$ of~$\bC[Z]$.
\end{definition}

That is, to prove that $z$ is semi-stable with respect to $\chi^u$ it is sufficient to check that $u \in \omega_{\bT}(z)$, hence we want to describe the orbit cones for considered action. We rely on a basic observation which follows directly from the definition of stability, see e.g.~\cite[8.1]{DolgachevLectInvTh}.

\begin{lemma}\label{lemma_stab_toric}
Fix a character~$\chi^u$, $u \in M_{\bT}$, which gives linearizations of actions of~$\,\bT$ both on $Z$ and on $A$. Then the sets of stable and semi-stable points with respect to~$\chi^u$ satisfy $Z^{ss} = Z \cap A^{ss}$ and $Z \cap A^s \subseteq Z^s$.
\end{lemma}

The first part can be rephrased in terms of orbit cones, cf.~\cite[2.5]{BerchtoldHausenGIT}.
\begin{corollary}
The orbit cone for $z \in Z$ and the action of~$\,\bT$ on~$Z$ is equal to the orbit cone of $z$ and the action on~$A$.
\end{corollary}

Orbit cones for the affine space~$A$ are easy to describe. Let $\pi \colon \hM \ra M_{\bT}$ be the homomorphism of lattices corresponding to $\bT \subseteq \bT_A$; we will assume that it is given by the matrix~$U$ of weights of the action of~$\bT$ on $A$. By $\gamma_z$ we denote the face of~$\hs^+$ generated by monomials non-vanishing on $\bT_A \cdot z \subset A$. The proof of the following statement is straightforward.

\begin{lemma}\label{lemma_orbit_cones_desc}
Orbit cones for the action of $\bT$ on~$A$, hence also on $Z$, are images of faces of~$\hs^+$ under~$\pi$. More precisely, $\omega_{\bT}(z) = \pi(\gamma_z)$.
\end{lemma}

\begin{corollary}\label{corollary_semistability_test}
A point $z$ is semi-stable with respect to the $\bT$-action on $Z$ linearized by~$\chi^u$ if and only if $u \in \pi(\gamma_z)$ (see also~\cite[Lem.~2.7]{BerchtoldHausenGIT}).
\end{corollary}

The next lemma follows from the fact that the actions of $\bT$ and $\bT_A$ on $A$ commute (or for semi-stability from~\ref{lemma_orbit_cones_desc}), cf.~\cite[2.5]{BerchtoldHausenGIT}.

\begin{lemma}\label{lemma_stable_orbits}
Stability, semi-stability and isotropy groups of the action of~$\,\bT$ on~$A$ (hence also semi-stability and isotropy groups of the action on $Z$) are invariants of~$\,\bT_A$, i.e.~are properties of whole $\bT_A$-orbits.
\end{lemma}

These observations are very useful in algorithms dealing with sets of semi-stable points, see~\cite{KeicherGITFan}.
Computations concerning stability are more subtle because of the orbit closedness condition: it may happen that a point is stable under the action on~$Z$, but not on~$A$. However, it turns out that for our purposes it is sufficient to check whether a point of $Z$ is in $A^s$ -- by~\ref{lemma_stab_toric} such points are stable in~$Z$.

\begin{lemma}\label{lemma_check_stability}
Consider a $\bT$-action on $A$ linearized by~$\chi^u$, and assume that the isotropy group of a point $z \in A$ is finite. If $u$ is in the relative interior of $\pi(\gamma_z)$, then $z$ is stable.
\end{lemma}

\begin{proof}
A point in the boundary of $\bT\cdot z$ is a limit of some one-parameter subgroup of~$\bT$, hence it belongs to an orbit corresponding to a proper face of~$\gamma_z$. Thus, to prove the stability of~$z$ we want to find a $\bT$-invariant section~$f$ of the trivial bundle on~$A$ such that $z \in A_f = \{a \in A \colon f(a) \neq 0\}$ and~$A_f$ does not contain any orbits corresponding to proper faces of~$\gamma_z$. We will choose~$f$ which is a character of~$\bT$ (regular on~$A$). If~$u$ is in the relative interior of $\pi(\gamma_z)$ then there is some $\overline{u}$ in the relative interior of $\gamma_z$ such that $\pi(\overline{u}) = u$, and we take $f = \chi^{\overline{u}}$. Because $\overline{u}$ is not contained in any face of~$\gamma_z$, then~$f$ vanishes on all orbits corresponding to faces of~$\gamma_z$.
\end{proof}

Next, we need to determine the orders of isotropy groups of points under the action of~$\bT$. They can be computed using the Smith normal form of a matrix, see~\cite[II.15]{Newman}, which is implemented e.g. in Singular,~\cite{Singular}. A matrix $U_z$ is obtained from the matrix~$U$ of weights of the $\bT$-action by choosing columns corresponding to non-zero coordinates in the orbit $\bT_A\cdot z$.

\begin{lemma}\label{lemma_snf}
Let $a_1,\ldots,a_p$ be the non-zero entries on the diagonal of the Smith normal form (over $\bZ$) of~$U_z$ for some~$z\in A$. If they fill the whole diagonal then the order of the isotropy group of~$z$ under the $\bT$-action is $a_1\cdots a_p$. If there are also zeroes on the diagonal then the isotropy group of~$z$ is infinite.
\end{lemma}

\begin{proof}
Note that columns of~$U_z$ are exactly the rays of the orbit cone~$\omega_{\bT}(z)$, i.e.~$U_z$ determines the homomorphism $M_z \ra M_{\bT}$ of monomial lattices of $\bT_A\cdot z$ and~$\bT$. Then $\Hom(M_{\bT}/M_z, \bC^*)$ is isomorphic to the kernel of the corresponding morphism of the tori, which is exactly the isotropy group of~$z$.
The Smith normal form of $U_z$ is obtained by multiplying it on both sides by some invertible integer matrices such that the result is a diagonal matrix with non-zero entries $a_1,\ldots,a_p$, where $a_i | a_{i+1}$ for $1 \leq i \leq p-1$. Thus it gives the description of the quotient group $M_{\bT}/M_z$ as a product of finite cyclic groups of orders $a_1,\ldots,a_p$ and $\bZ^q$, where~$q$ is the number of zeroes on the diagonal.
\end{proof}

Now we can describe the algorithm which we use to determine a good linearization for constructing a GIT quotient of~$Z$ by~$\bT$ explicitly. We are looking for a linearization~$\chi^u$ satisfying the following condition:

\refstepcounter{equation}\label{condition_good_lin}
\begin{flushleft}
\vspace{0.1cm}
\begin{tabular}{lm{11cm}}
  (\theequation) & \emph{semi-stability of a point of~$Z$ with respect to~$\chi^u$ implies its stability with respect to~$\chi^u$.} \\
\end{tabular}
\vspace{0.1cm}
\end{flushleft}

This is because in such a situation we obtain a geometric quotient together with a nice description of the set of stable points. By~\ref{lemma_stable_orbits} this set (and also the set of zeroes of the irrelevant ideal) is a sum of intersections of certain $\bT_A$-orbits in~$A$ with~$Z$.

\begin{definition}\label{def_R_relevant_orbits}
A~$\bT_A$-orbit in~$A$ which has non-empty intersection with~$Z$ and whose points are semi-stable with respect to a fixed linearization~$\chi^u$ will be called a $\bC[Z]$-relevant orbit with respect to~$\chi^u$. (We will skip the information about the linearization whenever the choice is clear.)
\end{definition}

Note that if a linearization~$\chi^u$ satisfies condition~\ref{condition_good_lin}, the intersections of $Z$ with all $\bC[Z]$-relevant orbits cover the set of stable points $Z^s$.

Algorithm~\ref{alg_stability} is implemented in the form of a small Singular package available at \href{http://www.mimuw.edu.pl/~marysia/gitcomp.lib}{www.mimuw.edu.pl/$\sim$marysia/gitcomp.lib}. The input data for the algorithm consists~of
\begin{enumerate}
\item the ideal~$\cI$ of~$Z$,
\item the matrix~$U$ defining the $\bT$-action on $A$ and $Z$,
\item a linearization of this action, given by a character $\chi^u$ of $\bT$, where $u \in M_{\bT}$.
\end{enumerate}

The output is the information whether $\chi^u$ satisfies condition~\ref{condition_good_lin}.

We start the computations from determining the set of $\bT_A$-orbits which have non-empty intersection with~$Z$. They are represented by convex polyhedral cones: faces of the positive orthant $\hs^+$ of $\hM$. Such cones are called \emph{$\cI$-faces} in~\cite{KeicherGITFan}. Note that by~\ref{lemma_stable_orbits} and~\ref{lemma_check_stability} we need to check only properties of the whole $\bT$-orbits, hence the program operates on lists of $\cI$-faces or corresponding orbit cones.

\begin{algorithm}\label{alg_stability}
The following actions are performed:

\begin{enumerate}[leftmargin=*]
\item Determine the list $\cF$ of $\cI$-faces of $\hs^+$.\\ Here we use a Singular package {\verb"GITfan.lib"} (see~\cite{KeicherGITFan}), which for given~$\cI$ returns the desired list of cones.
\item Determine semi-stable points.\\
    For all $\cI$-faces from~$\cF$ we compute rays of corresponding orbit cones using the matrix~$U$.
    Then, by~\ref{corollary_semistability_test} we check whether~$u$ is inside these cones.
    The result is the list $\cF^{ss}$ of orbit cones corresponding to $\bT_A$-orbits semi-stable with respect to~$\chi^u$.
\item Check finiteness of the isotropy group.\\
    The order of the isotropy group is computed for each cone from $\cF^{ss}$, as described in~\ref{lemma_snf}. If for all cones from $\cF^{ss}$ points of corresponding orbits have finite isotropy groups then the next step is performed. Otherwise the negative answer is given immediately.\\
    Note this point of the computations is independent of the linearization.
\item Check stability of elements of~$\cF^{ss}$.\\
    By~\ref{lemma_check_stability} we check whether~$u$ is in the relative interior of cones from~$\cF^{ss}$. If it is true for all cones from this list then the linearization given by~$\chi^u$ satisfies condition~\ref{condition_good_lin}. In this case the program can output the list~$\cF^{ss}$, which gives a useful description of the set of $\chi^u$-stable points of~$Z$. Also, the program returns the information on orders of isotropy groups for all stable orbits.
\end{enumerate}

\end{algorithm}

Finally, we reveal the main application of Algorithm~\ref{alg_stability}. We are looking for a linearization of the $\bT := \bT_\Lambda$ action on $Z := \Spec \cR$, embedded in~$A \simeq \bC^{15}$, which allows to describe explicitly the geometry of the quotient. A very good candidate, because of its symmetries, is~$\chi^{\kappa}$ given by the weight vector $\kappa = (2,2,2,2,2)$. The corresponding quotient makes a good starting point for performing flops leading to other resolutions, see Section~\ref{geom}. We prove that this is indeed a right choice.

\begin{proposition}\label{prop_good_linearization}
The linearization of the $\bT_\Lambda$-action on $\Spec \cR$ given by~$\chi^{\kappa}$ for $\kappa = (2,2,2,2,2)$ satisfies condition~\ref{condition_good_lin}. Therefore the corresponding quotient is geometric. Moreover, all points of $\Spec \cR$ which are semi-stable with respect to~$\chi^{\kappa}$ have trivial isotropy group.
\end{proposition}

\begin{proof}
Computations performed using the implementation of~Algorithm~\ref{alg_stability} give the result stated above. We use the weights of the $\bT_W$-action instead of~$\bT_\Lambda$, which changes just the order of the isotropy group multiplying it by~2.
\end{proof}

\subsection{The set of stable points}\label{section_stable_points}

Using Algorithm~\ref{alg_stability}, for a chosen linearization satisfying condition~\ref{condition_good_lin} one obtains an explicit description of the set~$(\Spec \cR)^s$ of stable points. In general, it comes in the form of the list $\cF^{ss}$ of $\cI$-faces corresponding to $\cR$-relevant orbits, see~\ref{def_R_relevant_orbits}.
We will describe these orbits in the case of the linearization $\chi^{\kappa}$; by~\ref{prop_good_linearization} their intersections with~$\Spec \cR$ cover the whole~$(\Spec \cR)^s$.

It turns out that for $\chi^{\kappa}$ there are only 167 $\cR$-relevant orbits in $A$. Because of the symmetries of generators of~$\cI$ the result may be presented as a much shorter list. Groups of these orbits are given by vanishing of sets of variables which differ by a certain permutation of indices, hence we list just combinatorial types of possible sets of vanishing variables, see Table~\ref{table_orbits}.

Note that in each description of an orbit type in Table~\ref{table_orbits} letters $a,b,c,d,e$ stand for different elements of $\{0,1,2,3,4\}$. The division into orbit types is based on the number of $u_i$'s equal to 0 and then on the set of $w_{ij}$'s equal to~0.

\medskip
\begin{table}[h]\caption{$\cR$-relevant orbits in the ambient affine space of $\Spec\cR$}
\label{table_orbits}
\begin{tabular}{c|c|c|c|c}
  type id & equations & $\#$ orbits & dim & dim orbit $\cap\Spec \cR$ \\
  \hline
  \hline
  5A & $\begin{array}{c} u_0=u_1=u_2=u_3=u_4=0 \\ w_{ab}=w_{cd}=0 \end{array}$ & 15 & 8 & 5 \\
  \hline
  5B & $\begin{array}{c} u_0=u_1=u_2=u_3=u_4=0 \\ w_{ab}=0 \end{array}$ & 10 & 9 & 6 \\
  \hline
  5C & $u_0=u_1=u_2=u_3=u_4=0$ & 1 & 10 & 7 \\
  \hline
  \hline
  3A & $\begin{array}{c} u_a=u_b=u_c=0 \\ w_{ab}=w_{ac}=w_{bc}=0 \\ w_{de}=0 \end{array}$ & 10 & 8 & 5 \\
  \hline
  3B & $\begin{array}{c} u_a=u_b=u_c=0 \\ w_{ab}=w_{de}=0 \end{array}$ & 30 & 10 & 6 \\
  \hline
  3C & $\begin{array}{c} u_a=u_b=u_c=0 \\ w_{de}=0 \end{array}$ & 10 & 11 & 7 \\
  \hline
  \hline
  1A & $\begin{array}{c} u_a=0 \\ w_{ab}=w_{ac}=w_{bc}=0 \\ w_{de}=0\end{array}$ & 30 & 10 & 6 \\
  \hline
  1B & $\begin{array}{c} u_a=0 \\ w_{bc}=w_{de}=0 \end{array}$ & 15 & 12 & 7 \\
  \hline
  1C & $\begin{array}{c} u_a=0 \\ w_{ab}=0 \end{array}$ & 20 & 13 & 7 \\
  \hline
  1D & $u_a=0$ & 5 & 14 & 8 \\
  \hline
  \hline
  0A & $w_{ab}=w_{ac}=w_{bc}=w_{de}=0$ & 10 & 11 & 7 \\
  \hline
  0B & $w_{ab}=0$ & 10 & 14 & 8 \\
  \hline
  0C &  & 1 & 15 & 9 \\
\end{tabular}
\end{table}

\subsection{Smoothness of the quotient}\label{section_singularities}
The last element needed in the proof that the geometric quotient $X = (\Spec \cR)^s / \bT_\Lambda$ associated with the distinguished linearization $\chi^{\kappa}$ is a resolution of singularities of $V/G$ is the smoothness of~$X$. We follow the idea explained in~\cite[Prop.~4.5]{DontenSurfaces}: we check that~$X$ is a geometric quotient of a smooth variety by a free torus action. Since by~\ref{prop_good_linearization} the action of~$\bT_\Lambda$ on~$(\Spec \cR)^s$ is free, it is sufficient to show that $X$ is nonsingular.

A natural approach is to compute the ideal of the set of singular points of $\Spec \cR$ directly from the Jacobian criterion and show that it has empty intersection with $(\Spec \cR)^s$. However, the input data is too big for performing a direct computation in reasonable time. Hence we divide the process into a few separate cases and make use of the $\bT_\Lambda$ action and the description of $(\Spec \cR)^s$ in terms of the toric structure of the ambient affine space $\bC^{15}$ in Table~\ref{table_orbits}.

We rely on two basic observations. First, it is sufficient to prove the smoothness of one point in every $\bT_\Lambda$-orbit in $(\Spec \cR)^s$. Thus we may consider only points with all $u_i$ equal to~0 or~1, that is, using the $\bT_\Lambda$-action we move non-zero $u_i$'s to~1. This already simplifies the Jacobian matrix of $\Spec \cR$ a lot. Then, since we do not want to treat each orbit separately, we use the symmetries of equations of $\Spec \cR$ so that we can consider certain representatives of combinatorial types of $\cR$-relevant orbits. The following observation can be checked straightforwardly.

\begin{lemma}\label{equations-symmetry}
The equations of $\Spec \cR$, listed in~\ref{ideal-gens}, are invariant under a cyclic change of indices $i \mapsto i+1 \mod 5$.
\end{lemma}

\begin{proposition}\label{quotient-smoothness}
The set $(\Spec \cR)^s$ of stable points with respect to the linearization $\chi^{\kappa}$ is nonsingular.
\end{proposition}

\begin{proof}
The argument is computational. We explain how to deal with the computations using basic functions of, for example, Macaulay2, \cite{M2}. We assume that the equations of $\Spec \cR$ are ordered as in~\ref{ideal-gens}. To compute the Jacobian matrix we differentiate with respect to variables ordered as follows:
$$w_{01},w_{02},w_{03},w_{04},w_{12},w_{13},w_{14},w_{23},w_{24},w_{34},u_0,u_1,u_2,u_3,u_4.$$

Take $z \in \Spec \cR$ and assume that all $6\times 6$ minors of the Jacobian matrix vanish at $z$. Then we have to show that $z \notin (\Spec \cR)^s$, that is $z$ does not belong to any orbit from Table~\ref{table_orbits}. Let us outline the computations. For each orbit type in Table~\ref{table_orbits} we choose representatives with respect to the cyclic group action, using~\ref{equations-symmetry}, and consider only $z$ from these chosen orbits. We simplify the Jacobian matrix substituting 0 or 1 for some variables, looking at the $\bT_\Lambda$-action on~$z$ and its orbit type. Then we compute some (suitably chosen) $6\times 6$ minors of this matrix and look for monomials. Finding a monomial minor means that the product of some coordinates vanishes at~$z$. Usually this gives a few subcases to consider (vanishing of each coordinate from the product has to be considered separately). However, we obtain more precise information of the orbit type of~$z$, which simplifies the Jacobian matrix even more. In each case, after a small number of such steps, we arrive at the conclusion that $z$ is not contained in any $\cR$-relevant orbit, which finishes the proof.

We are left with providing the details of the computations. To shorten the description, by $\det(i_1,\ldots,i_k | j_1,\ldots,j_k)$ we will denote the minor of the rows $i_1,\ldots,i_k$ and the columns $j_1,\ldots,j_k$ of the Jacobian matrix of equations of $\Spec \cR$. By $\mon(x_{k_1},\ldots,x_{k_n})$ we understand the set of all monomials in variables $x_{k_1},\ldots,x_{k_n}$. There are four cases depending on the type of the orbit from Table~\ref{table_orbits} in which~$z$ lies.

\noindent\textbf{Type 5.} We have $u_0=u_1=u_2=u_3=u_4=0$. After substituting into the Jacobian matrix check that $\det(7,\ldots,12 | 2,3,4,5,8,15) \in \mon(w_{01}, w_{02}, w_{12})$, hence one of these variables is 0. By permuting indices we get two cases.
    \begin{enumerate}[leftmargin=*]
    \item[(a)] $w_{01} = 0$.\\ Then we have $\det(0,1,4,11,13,14 | 0,1,2,9,11,13) \in \mon(w_{13}, w_{14}, w_{34})$, and $\det(5,6,9,10,12,13 | 1,3,4,7,8,14) \in \mon(w_{02}, w_{03}, w_{23})$. Hence at least 3 of $w_{ij}$'s are 0, which is impossible in orbits of type 5 in Table~\ref{table_orbits}.
    \item[(b)] $w_{02} = 0$.\\ Then we have $\det(0,1,4,10,13,14 | 0,1,2,6,7,9) \in \mon(w_{03}, w_{04}, w_{34})$, and also $\det(2,3,9,11,12,14 | 0,3,4,10,11,14) \in \mon(w_{12}, w_{14}, w_{24})$. Hence again at least three $w_{ij}$'s vanish.
    \end{enumerate}

\noindent\textbf{Type 3.} Applying the $\bT_\Lambda$ action we may move to the point where these $u_i$'s which are nonzero are equal to~1. By Remark~\ref{equations-symmetry} there are two cases.
    \begin{enumerate}[leftmargin=*]
    \item[(a)] $u_0=u_1=1$, $u_2=u_3=u_4=0$, $w_{01} = 0$.\\ Then $\det(0,\ldots,4,14 | 0,1,2,9,17,18) \in \mon(w_{04},w_{34})$, but from the equations for cases 3A and 3B we see that only $w_{34}=0$ could happen.\\ Next, $\det(0,\ldots,3,6,14 | 0,1,3,10,18,19)$ is a monomial, so at least three $w_{ij}$'s are~0. This means that we are in the case 3A and $w_{23}=w_{24}=0$. However, in this case $\det(0,\ldots,3,7,12 | 0,9,10,11,17,18)$ is a monomial and too many variables vanish.
    \item[(b)] $u_0=u_2=1$, $u_1=u_3=u_4=0$, $w_{02} = 0$.\\ Then $\det(0,\ldots,4,11 | 0,1,2,9,15,18)\in \mon(w_{04}, w_{12}, w_{34})$. Again, the only possibility consistent with equations of 3A and 3B is $w_{34}=0$.\\ Now $\det(0,\ldots,3,5,11 | 0,9,10,11,17,18) \in \mon(w_{01}, w_{04}, w_{24})$, but none of these variables can be 0 in the cases of type~3.
    \end{enumerate}

\noindent\textbf{Type 1.} Permuting indices and applying the $\bT_\Lambda$ action we may assume that $u_0=u_1=u_2=u_3=1$ and $u_4=0$. Then $\det(5,7,\ldots,11 | 2,3,4,5,17,19)\in \mon(w_{01}, w_{03})$ and $\det(0,2,6,9,10,12 | 1,3,5,12,14,15)\in \mon(w_{02}, w_{12})$. Hence at least two variables vanish and we are in the case~1A or~1B. From their description we see that vanishing variables are $w_{03}$ and $w_{12}$ -- two vanishing variables without index~4 must have disjoint set of indices. Then $\det(4,7,\ldots,11 | 2,3,4,8,17,19) = w_{01}^8$, but $w_{03}=w_{12}=w_{01}=0$ is impossible for any of these types.

\noindent\textbf{Type 0.} Applying the $\bT_\Lambda$ action we may assume that $u_0=u_1=u_2=u_3=u_4=1$. Then $\det(6,\ldots,11 | 2,3,4,5,11,17)\in \mon(w_{01}, w_{04})$. Since all $u_i$ take the same value, the situation is symmetric with respect to the cyclic permutation of indices. Using permutations one can produce four other monomials in two variables from the given one and check that thus at least three different variables vanish. Hence we are in the case of type~0A and up to a cyclic permutation there are two possibilities:
    \begin{enumerate}[leftmargin=*]
    \item[(a)] $w_{01}=w_{02}=w_{12}=w_{34}=0$.\\ Then $\det(0,1,2,3,5,7 | 5,10,11,14,18,19) \in \mon(w_{04},w_{13})$, so five variables vanish, which is impossible in orbits of type 0A.
    \item[(b)] $w_{01}=w_{13}=w_{03}=w_{24}=0$.\\ Then $\det(0,\ldots,4,9 | 0,1,2,9,16,18) \in \mon(w_{04},w_{34})$, a contradiction again.
    \end{enumerate}
\end{proof}

\begin{theorem}\label{thmGITquotient}
The GIT quotient $X$ of $\Spec \cR$ by~$\bT_\Lambda$ associated with the linearization $\chi^{\kappa}$ is a resolution of singularities of~$V/G$.
\end{theorem}

\begin{proof}
The isomorphism $\cR^{\bT_\Lambda} \simeq \bC[V]^G$ proved in~\ref{map-of-invariants} gives a proper birational morphism from~$X$ to~$V/G$. The properness follows by~\cite[14.1.12]{CoxLittleSchenck} applied to the embedding $\Spec \cR$ and its quotients in the toric ambient spaces. And by~\ref{quotient-smoothness} we know that~$X$ is smooth.
\end{proof}


\section{The geometry of resolutions}\label{geom}

\subsection{The central resolution}\label{central-resolution}
Let us summarize the information which we get from previous sections;
$G\subset Sp(V)$ is the group defined in \ref{the-group}.  The
exceptional set of the resolution $\varphi: X\rightarrow V/G$
constructed as a GIT quotient in the previous section is covered by
divisors $E_0,\dots, E_4$ associated to the classes of symplectic
reflections in $G$. Each $E_i$ is contracted by $\varphi$ to a surface
of $A_1$ singularities outside of $[0]\in V/G$.  In terms of the ring
$\cR$ the divisors $E_i$ are associated to functions $t_i^{-2}$,
\ref{inclusion}.  By \ref{sympl-resolution-form} the resolution
$\varphi: X\rightarrow V/G$ is symplectic.  There is a unique
2-dimensional fiber of $\varphi$ over $[0]\in V/G$ which has 11
components, \ref{KaledinMcKay}, \ref{group}.

By $C_i$ we denote a general fiber of $\varphi_{|E_i}$. Clearly
$E_i\cdot C_j$ is $-2$ if $i=j$ and it is zero otherwise.  Now we
define $\kappa=\sum_i e_i$. In terms of the basis in
$\NU(X)=\Cl(X)\otimes\bR$ dual to classes of $C_i$'s the class
$\kappa$ is the vector $(2,2,2,2,2)$, see \ref{class-sequence-1}. For
$i=0,\dots,4$ in $\NU(X)$ we consider classes $e_i=[-E_i]$. By
\cite[Thm.~3.5]{AW2014} we get the following.

\begin{lemma}\label{mov-cone}
  For every resolution $X\rightarrow V/G$ the cone of movable divisors
  $\Mov(X)$ is spanned by the classes $e_i$.
\end{lemma}

By Theorem \ref{thmGITquotient} the GIT quotient of $\cR$ by
$\bT_\Lambda$ associated to the character $\kappa$ is a resolution
$\varphi^\kappa: X^\kappa\ra V/G$.

Recall that Table \ref{table_orbits} presents a list of big torus
orbits in the affine space containing $\Spec\cR$ which are relevant
with respect to the finite isotropy and the semistability condition
associated to $\kappa$. Note that divisors $E_i$ are associated to
relevant orbits of type 1D. The intersection $\bigcap_i E_i$ is
associated to the unique relevant orbit of type 5C. In fact, from
\ref{emmbeddP^2_4} we see that this special orbit comes from an
equivariant embedding $\Spec\cR(\bP^2_4)\hookrightarrow\Spec\cR$.

This gives rise to an embedding of GIT quotients $\iota:
\bP^2_4\hookrightarrow X^\kappa$ such that $\iota^*: \Pic X^\kappa
\rightarrow \Pic \bP^2_4$ is an isomorphism. By $F_0$ we will denote
$\iota(\bP^2_4)$. It follows that we can identify
$\NU(X^\kappa)=\NU(\bP^2_4)$ and we have $\Nef(X^\kappa)\subseteq
cone(\alpha_i,\ \beta_i: 0\leq i\leq 4)$, where
$\alpha_i:=(e_i+\kappa)/2$ and $\beta_i:=(-e_i+\kappa)/2$
c.f.~\ref{Nef-Eff-P^2_4}.

Dually, we have isomorphism $\iota_*: \Nu(X)\iso\Nu(\bP^2_4)=
\NU(\bP^2_4)$ and via this identification the classes of $(-1)$ curves
on $F_0\iso\bP^2_4$ are $f_{ij}=(e_i+e_j)/2$. Let $C_{ij}\subset F_0$
be one of these $(-1)$-curves. Then the family of deformations of
$C_{ij}$ is of dimension 2 at least, see
\cite[2.3]{WierzbaWisniewski}, and it must cover a component of a
2-dimensional fiber of $\varphi^\kappa$. Let us call such a component
$F_{ij}$. Since intersection of $C_{ij}$ with the ample class $\kappa$
is 1, the family of deformations of $C_{ij}$ in $F_{ij}$ is unsplit
and of dimension 2, hence every curve in $F_{ij}$ is numerically
proportional to $C_{ij}$, see e.g.~\cite[IV.3.13.3]{KollarRatCurves}.

Because the 2-dimensional fiber of $\varphi^\kappa$ has 11 components
we have a bijection between $C_{ij}$'s and components of this fiber
different from $F_0$. Also, it follows that all curves in the
2-dimensional fiber of $\varphi^\kappa$ have classes in $\Eff(F_0)$
and therefore, dually, $\Nef(X) = \Nef(F_0)$. Therefore a contraction
of the $(-1)$ curve $C_{ij}$ in $F_0$ extends to a small contraction
of $X^\kappa$ and by \cite[Thm.~1.1]{WierzbaWisniewski}
$F_{ij}\iso\bP^2$. Again, because $C_{ij}$ has intersection 1 with the
ample class it follows that it is a line on $F_{ij}$. Thus we have
proved the following.

\begin{proposition}\label{special-resolution}
  There exists a resolution $X^\kappa\rightarrow V/G$ such that
  $\kappa$ is a class of an ample divisor on $X^\kappa$.  The unique
  2-dimensional fiber of the resolution $X^\kappa \rightarrow V/G$
  consists of 11 components:
  \begin{itemize}
  \item the unique component  $F_0=\bigcap_i E_i\iso\bP^2_4$
  \item 10 components $F_{ij}$, for $0\leq i<j\leq 4$, which are
    contained in intersections of $E_k$'s such that $k\not\in\{i,j\}$;
    $F_{ij}\iso\bP^2$
  \end{itemize}
  The intersection $F_0\cap F_{ij}$ is a line on $F_{ij}$ and
  $(-1)$-curve on $F_0$. The line bundle associated to $\kappa$ is
  $-K_{F_0}$ on $F_0$ and $\cO(1)$ on every $F_{ij}$.
\end{proposition}

We note that curves $C_{ij}$ can be related to orbits of type 5B while
the components $F_{ij}$ can be related to orbits of type 3C in Table
\ref{table_orbits}. In fact, the arguments above regarding $F_{ij}$'s
can be replaced by direct calculations of quotients of respective
closed subsets of $\Spec\cR$.

\subsection{Flops.}
We can use identification $\iota^*$ introduced in the previous section
to describe the other resolutions of $V/G$ which will be obtained from
$X^\kappa\rightarrow V/G$ by Mukai flops, see \cite{WierzbaWisniewski}.

\begin{figure}
\caption{Flops of symplectic resolutions of $V/G$}
\label{flops-fig}
$$\begin{array}{ccc}
F_0=\bP^2_4&&F_0=\bP^2_3\\ \\
\begin{xy}<45pt,0pt>:
(0,1)*={\pb}="01", (0,1.2)*={F_{01}},
(0.588,0.809)*={\pb}="02", (0.6,1)*={F_{02}}, (0.95,0.27)*{\pb}="12", (1.1,0.45)*={F_{12}},
(-0.588,0.809)*={\pb}="14", (-0.6,1)*={F_{14}}, (-0.95,0.27)*={\pb}="04", (-1.1,0.45)*={F_{04}},
(0.588,-0.809)*={\pb}="23",  (0.6,-1)*={F_{23}}, (0.95,-0.27)*={\pb}="13", (1.1,-0.45)*={F_{13}},
(-0.588,-0.809)*={\pb}="34", (-0.6,-1)*={F_{34}}, (-0.95,-0.27)*={\pb}="03",  (-1.1,-0.45)*={F_{03}},
(0,-1)*={\pb}="24",  (0,-1.2)*={F_{24}},
(1.5,0.9)*={\pbfour}="0", (1.7,1.1)*={F_{0}},
"01";"24" **@{.}, "01";"23" **@{.}, "01";"34" **@{.},
"02";"34" **@{.}, "02";"13" **@{.}, "02";"14" **@{.},
"03";"12" **@{.}, "03";"14" **@{.}, "03";"24" **@{.},
"04";"12" **@{.}, "04";"13" **@{.}, "04";"23" **@{.},
"12";"34" **@{.}, "13";"24" **@{.}, "14";"23" **@{.},
"0";"01" **@{-}, "0";"02" **@{-}, "0";"03" **@{-}, "0";"04" **@{-},
"0";"12" **@{-}, "0";"13" **@{-}, "0";"14" **@{-},
"0";"23" **@{-}, "0";"24" **@{-}, "0";"34" **@{-},
\end{xy}&{\color{red}\longrightarrow}&
\begin{xy}<45pt,0pt>:
(0,1)*={\pb}="01", (0,1.2)*={F_{01}},
(0.588,0.809)*={\pb}="02", (0.6,1)*={F_{02}}, (0.95,0.27)*{\pbone}="12", (1.1,0.45)*={F_{12}},
(-0.588,0.809)*={\pb}="14", (-0.6,1)*={F_{14}}, (-0.95,0.27)*={\pb}="04", (-1.1,0.45)*={F_{04}},
(0.588,-0.809)*={\pbone}="23",  (0.6,-1)*={F_{23}}, (0.95,-0.27)*={\pbone}="13", (1.1,-0.45)*={F_{13}},
(-0.588,-0.809)*={\pb}="34", (-0.6,-1)*={F_{34}}, (-0.95,-0.27)*={\pb}="03",  (-1.1,-0.45)*={F_{03}},
(0,-1)*={\pb}="24",  (0,-1.2)*={F_{24}},
(1.5,0.9)*={\pbfour}="0", (1.7,1.1)*={F_{0}},
"01";"24" **@{.}, "01";"23" **@{.}, "01";"34" **@{.},
"02";"34" **@{.}, "02";"13" **@{.}, "02";"14" **@{.},
"03";"12" **@{.}, "03";"14" **@{.}, "03";"24" **@{.},
"04";"12" **@{-}, "04";"13" **@{-}, "04";"23" **@{-},
"12";"34" **@{.}, "13";"24" **@{.}, "14";"23" **@{.},
"0";"01" **@{-}, "0";"02" **@{-}, "0";"03" **@{-}, "0";"04" **@{.},
"0";"12" **@{-}, "0";"13" **@{-}, "0";"14" **@{-},
"0";"23" **@{-}, "0";"24" **@{-}, "0";"34" **@{-},
\end{xy}\\&{\color{red}\swarrow}&\\
F_0=\bP^2_2&&F_0=\bP^1\times\bP^1
\\ \\
\begin{xy}<45pt,0pt>:
(0,1)*={\pb}="01", (0,1.2)*={F_{01}},
(0.588,0.809)*={\pb}="02", (0.6,1)*={F_{02}}, (0.95,0.27)*{\pbtwo}="12", (1.1,0.45)*={F_{12}},
(-0.588,0.809)*={\pbone}="14", (-0.6,1)*={F_{14}}, (-0.95,0.27)*={\pb}="04", (-1.1,0.45)*={F_{04}},
(0.588,-0.809)*={\pbone}="23",  (0.6,-1)*={F_{23}}, (0.95,-0.27)*={\pbone}="13", (1.1,-0.45)*={F_{13}},
(-0.588,-0.809)*={\pb}="34", (-0.6,-1)*={F_{34}}, (-0.95,-0.27)*={\pb}="03",  (-1.1,-0.45)*={F_{03}},
(0,-1)*={\pbone}="24",  (0,-1.2)*={F_{24}},
(1.5,0.9)*={\pbfour}="0", (1.7,1.1)*={F_{0}},
"01";"24" **@{.}, "01";"23" **@{.}, "01";"34" **@{.},
"02";"34" **@{.}, "02";"13" **@{.}, "02";"14" **@{.},
"03";"12" **@{-}, "03";"14" **@{-}, "03";"24" **@{-},
"04";"12" **@{-}, "04";"13" **@{-}, "04";"23" **@{-},
"12";"34" **@{.}, "13";"24" **@{.}, "14";"23" **@{.},
"0";"01" **@{-}, "0";"02" **@{-}, "0";"03" **@{.}, "0";"04" **@{.},
"0";"12" **@{-}, "0";"13" **@{-}, "0";"14" **@{-},
"0";"23" **@{-}, "0";"24" **@{-}, "0";"34" **@{-},
\end{xy}&{\color{red}\longrightarrow}&
\begin{xy}<45pt,0pt>:
(0,1)*={\pbone}="01", (0,1.2)*={F_{01}},
(0.588,0.809)*={\pbone}="02", (0.6,1)*={F_{02}}, (0.95,0.27)*={\pbthreecoll}="12", (1.1,0.45)*={F_{12}},
(-0.588,0.809)*={\pbone}="14", (-0.6,1)*={F_{14}}, (-0.95,0.27)*={\pb}="04", (-1.1,0.45)*={F_{04}},
(0.588,-0.809)*={\pbone}="23",  (0.6,-1)*={F_{23}}, (0.95,-0.27)*={\pbone}="13", (1.1,-0.45)*={F_{13}},
(-0.588,-0.809)*={\pb}="34", (-0.6,-1)*={F_{34}}, (-0.95,-0.27)*={\pb}="03",  (-1.1,-0.45)*={F_{03}},
(0,-1)*={\pbone}="24",  (0,-1.2)*={F_{24}},
(1.5,0.9)*={\pbfour}="0", (1.7,1.1)*={F_{0}},
"01";"24" **@{.}, "01";"23" **@{.}, "01";"34" **@{-},
"02";"34" **@{-}, "02";"13" **@{.}, "02";"14" **@{.},
"03";"12" **@{-}, "03";"14" **@{-}, "03";"24" **@{-},
"04";"12" **@{-}, "04";"13" **@{-}, "04";"23" **@{-},
"12";"34" **@{-}, "13";"24" **@{.}, "14";"23" **@{.},
"0";"01" **@{-}, "0";"02" **@{-}, "0";"03" **@{.}, "0";"04" **@{.},
"0";"12" **@{-}, "0";"13" **@{-}, "0";"14" **@{-},
"0";"23" **@{-}, "0";"24" **@{-}, "0";"34" **@{.},
\end{xy}

\\{\color{red}\downarrow}\\
F_0=\bP^2_1\\ \\
\begin{xy}<45pt,0pt>:
(0,1)*={\pb}="01", (0,1.2)*={F_{01}},
(0.588,0.809)*={\pb}="02", (0.6,1)*={F_{02}}, (0.95,0.27)*{\pbtwo}="12", (1.1,0.45)*={F_{12}},
(-0.588,0.809)*={\pbtwo}="14", (-0.6,1)*={F_{14}}, (-0.95,0.27)*={\pb}="04", (-1.1,0.45)*={F_{04}},
(0.588,-0.809)*={\pbone}="23",  (0.6,-1)*={F_{23}}, (0.95,-0.27)*={\pbtwo}="13", (1.1,-0.45)*={F_{13}},
(-0.588,-0.809)*={\pbone}="34", (-0.6,-1)*={F_{34}}, (-0.95,-0.27)*={\pb}="03",  (-1.1,-0.45)*={F_{03}},
(0,-1)*={\pbone}="24",  (0,-1.2)*={F_{24}},
(1.5,0.9)*={\pbfour}="0", (1.7,1.1)*={F_{0}},
"01";"24" **@{.}, "01";"23" **@{.}, "01";"34" **@{.},
"02";"34" **@{-}, "02";"13" **@{-}, "02";"14" **@{-},
"03";"12" **@{-}, "03";"14" **@{-}, "03";"24" **@{-},
"04";"12" **@{-}, "04";"13" **@{-}, "04";"23" **@{-},
"12";"34" **@{.}, "13";"24" **@{.}, "14";"23" **@{.},
"0";"01" **@{-}, "0";"02" **@{.}, "0";"03" **@{.}, "0";"04" **@{.},
"0";"12" **@{-}, "0";"13" **@{-}, "0";"14" **@{-},
"0";"23" **@{-}, "0";"24" **@{-}, "0";"34" **@{-},
\end{xy}\\{\color{red}\downarrow} \\
F_0=\bP^2&&F_0=(\bP^2)^\vee\\ \\
\begin{xy}<45pt,0pt>:
(0,1)*={\pb}="01", (0,1.2)*={F_{01}},
(0.588,0.809)*={\pb}="02", (0.6,1)*={F_{02}}, (0.95,0.27)*{\pbtwo}="12", (1.1,0.45)*={F_{12}},
(-0.588,0.809)*={\pbtwo}="14", (-0.6,1)*={F_{14}}, (-0.95,0.27)*={\pb}="04", (-1.1,0.45)*={F_{04}},
(0.588,-0.809)*={\pbtwo}="23",  (0.6,-1)*={F_{23}}, (0.95,-0.27)*={\pbtwo}="13", (1.1,-0.45)*={F_{13}},
(-0.588,-0.809)*={\pbtwo}="34", (-0.6,-1)*={F_{34}}, (-0.95,-0.27)*={\pb}="03",  (-1.1,-0.45)*={F_{03}},
(0,-1)*={\pbtwo}="24",  (0,-1.2)*={F_{24}},
(1.5,0.9)*={\pbfour}="0", (1.7,1.1)*={F_{0}},
"01";"24" **@{-}, "01";"23" **@{-}, "01";"34" **@{-},
"02";"34" **@{-}, "02";"13" **@{-}, "02";"14" **@{-},
"03";"12" **@{-}, "03";"14" **@{-}, "03";"24" **@{-},
"04";"12" **@{-}, "04";"13" **@{-}, "04";"23" **@{-},
"12";"34" **@{.}, "13";"24" **@{.}, "14";"23" **@{.},
"0";"01" **@{.}, "0";"02" **@{.}, "0";"03" **@{.}, "0";"04" **@{.},
"0";"12" **@{-}, "0";"13" **@{-}, "0";"14" **@{-},
"0";"23" **@{-}, "0";"24" **@{-}, "0";"34" **@{-},
\end{xy}&{\color{red}\longrightarrow}&
\begin{xy}<45pt,0pt>:
(0,1)*={\pbone}="01", (0,1.2)*={F_{01}},
(0.588,0.809)*={\pbone}="02", (0.6,1)*={F_{02}}, (0.95,0.27)*{\pquadric}="12", (1.1,0.45)*={F_{12}},
(-0.588,0.809)*={\pquadric}="14", (-0.6,1)*={F_{14}}, (-0.95,0.27)*={\pbone}="04", (-1.1,0.45)*={F_{04}},
(0.588,-0.809)*={\pquadric}="23",  (0.6,-1)*={F_{23}}, (0.95,-0.27)*={\pquadric}="13", (1.1,-0.45)*={F_{13}},
(-0.588,-0.809)*={\pquadric}="34", (-0.6,-1)*={F_{34}}, (-0.95,-0.27)*={\pbone}="03",  (-1.1,-0.45)*={F_{03}},
(0,-1)*={\pquadric}="24",  (0,-1.2)*={F_{24}},
(1.5,0.9)*={\pbfour}="0", (1.7,1.1)*={F_{0}},
"01";"24" **@{-}, "01";"23" **@{-}, "01";"34" **@{-},
"02";"34" **@{-}, "02";"13" **@{-}, "02";"14" **@{-},
"03";"12" **@{-}, "03";"14" **@{-}, "03";"24" **@{-},
"04";"12" **@{-}, "04";"13" **@{-}, "04";"23" **@{-},
"0";"01" **@{-}, "0";"02" **@{-}, "0";"03" **@{-}, "0";"04" **@{-},
"0";"12" **@{.}, "0";"13" **@{.}, "0";"14" **@{.},
"0";"23" **@{.}, "0";"24" **@{.}, "0";"34" **@{.},
\end{xy}
\end{array}$$
\end{figure}
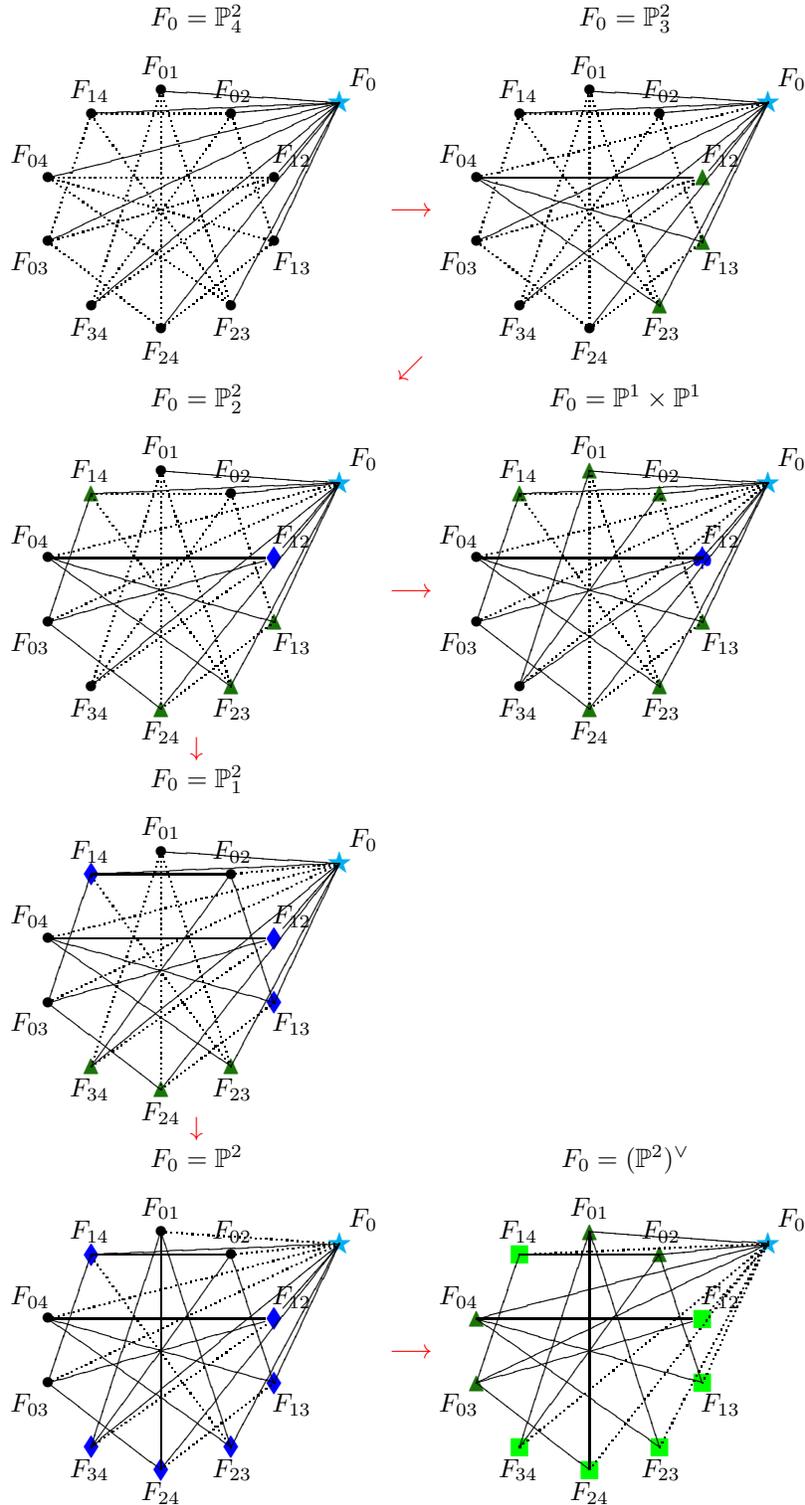

Our description is similar to that of \cite[Sect.~6.7]{AW2014}.
Figure \ref{flops-fig} illustrates the components of the 2-dimensional fiber of
these resolutions and their incidence.  The distinguished {\em
  central} component $F_0$ is always denoted by $\pbfour$. Each of the
diagrams in the table is described by the isomorphism class of this
component.  By $(\bP^2)^\vee$ we denote the central component after
the flop. Also the other components of the 2-dimensional fiber are
denoted by the same name $F_{ij}$ for every resolution.  Their
isomorphisms types are described as follows: $\pb=\bP^2$,
$\pbone=\bP^2_1$, $\pbtwo=\bP^2_2$, $\pquadric= \bP^1 \times \bP^1$
and $\pbthreecoll=\bP^2_{\overline{3}}$ is the blow-up of $\bP^2$ in
three collinear points.

The incidence of components is denoted by line segments joining the
respective symbols. The solid line denotes intersection along a
rational curve while a dotted line denotes intersection at a
point. For the sake of clarity we ignore the intersection (at a point)
of components which will not be flopped.

The first diagram in Figure \ref{flops-fig} illustrates the special fiber
of the unique central resolution in which the central component is
$\bP^2_4$ and the remaining components are $\bP^2$. The other
resolutions are obtained by flopping some components which are
isomorphic to $\bP^2$, the direction of the flops is indicated by
arrows. That is, via the identification of $\NU(X)$ with
$\NU(F_0)=\NU(\bP^2_4)$, the ample cone of $\bP^2_4$ is placed in the
center of the movable cone of $X$ and the direction of our flops
points out outside the central chamber.

The first two flops are along $F_{04}$ and $F_{03}$ and they lead to
the central component isomorphic to, respectively, $\bP^2_3$ (there
are 10 different resolutions of this type) and $\bP^2_2$ (there are 30
different resolutions of this type). The surface $\bP^2_2$ has three
$(-1)$-curves which make a chain, contracting the central one we get
$\bP^1\times\bP^1$ while contracting any of the other two we get
$\bP^2_1$. Respectively, we can consider either a flop along $F_{24}$
which makes $F_0$ isomorphic to $\bP^1\times\bP^1$ (there are 10
resolutions of this type) or flop along $F_{02}$ and get $F_0=\bP^2_1$
(there are 20 resolutions of this type). This latter one type can be
further flopped along $F_{01}$ to get $F_0=\bP^2$ (5 resolutions of
this type). Finally, $F_0$ can be flopped, we denote it then by
$(\bP^2)^\vee$; there are 5 resolutions of this type we will denote
them $X^i\rightarrow V/G$. The cone $\Nef(X^i)$ is simplicial and it
is generated by $e_i$ and by four classes $(e_i+e_j)/2$, with $j\ne
i$; we will call it an {\em outer} chamber of $\Mov(X)$.

Note that, after the first flop, the surfaces $F_{23}$, $F_{24}$ and
$F_{34}$ have a common point, which comes from contracting the $(-1)$
curve on $F_0$. We ignore it in our diagram since none of these three
components will be flopped. Similarly, we will not put in the diagrams
the intersection points which will be negligible from the point of
view of possible flops.

Counting the number of resolutions we obtain a result already announced
by Bellamy, \cite{Bellamy81}.

\begin{proposition}\label{81resolutions}
  There are 81 symplectic resolutions of the quotient singularity
  $V/G$.
\end{proposition}

In fact, from our construction it follows that symplectic resolutions
of $V/G$ are in bijection with chambers in the cone
$cone(e_0,\dots,e_4)$ obtained by cutting this cone with hyperplanes
perpendicular to classes $f_{ij}$ and $\beta_i$ which were defined in
subsection \ref{P^2_4}.


\section{A Kummer 4-fold}

\subsection{A Kummer surface}\label{Kummer-surface}
Let $\bE$ be an elliptic curve with the complex multiplication by
$i=\sqrt{-1}$. That is we have a linear automorphism $i\in
Aut(\bE)$ such that $i^2=-id_\bE$. For simplicity, we can assume
$\bE=\bC/\bZ[i]$ and $i$ acts by the standard complex
multiplication. The automorphism $i$ of $\bE$ has two fixed points
$p_0=[0]$ and $p_1=[(1+i)/2]$ while it interchanges the other two
order 2 points $i[1/2]=[i/2]$ and $i[i/2]=[1/2]$; here the square
brackets denote the classes in $\bC/\bZ[i]$. We see that, in fact,
this multiplication yields an isomorphism of the group of order 2
points on $\bE$ with the ring $\bZ_2[i]$ with $p_1$ identified with the
unique zero divisor $1+i$.

Let us recall that the standard representation of the binary dihedral
group of order 8, or the quaternion group, $Q_8$ in $SL(2,\bC)$ is
given by the following $2\times 2$ matrices over $\bC$:
$$\begin{array}{ccc}
A_I=\left(\begin{array}{cc}i&0\\0&-i\end{array}\right)&
A_J=\left(\begin{array}{cc}0&1\\-1&0\end{array}\right)&
A_K=\left(\begin{array}{cc}0&i\\i&0\end{array}\right)
\end{array}$$

Since, in fact, this representation is in $SL(2,\bZ[i])$, we have the
action of $Q_8$ on $\bE^2$. It is not hard to check (see the argument
below) that the group $Q_8$ acts on $\bE^2$ with 2 fixed points,
namely $(p_0,p_0)$ and $(p_1,p_1)$, 6 point with isotropy $\bZ_4$ and
8 points with isotropy $\bZ_2$. This makes 16 points with non-trivial
isotropy. In fact, they are all order 2 points on $\bE^2$, because
$-id$ is contained in every nontrivial subgroup of $Q_8$.

If $n_8$, $n_4$ and $n_2$ are the number of orbits of points with the
isotropy $Q_8$, $\bZ_4$ and $\bZ_2$, respectively, then taking into
consideration the ranks of the normalizers of these subgroups we get
$n_8=2$, $n_4=3$, $n_2=2$ and $n_8+2n_4+4n_2=16$. Thus on the quotient
$\bE^2/Q_8$ we have $n_8=2$ singularities of type $D_4$ and $n_4=3$
singularities of type $A_3$ and $n_2=2$ singular points of type
$A_1$. Resolving these singularities we obtain a K3 surface with
$4\cdot n_8+ 3\cdot n_4+ 1\cdot n_2 = 19$ exceptional
$(-2)$-curves. We note that $\dim_\bC H^{1,1}(\bE^2)^{Q_8}=1$ completes
this number to the dimension of $H^{1,1}$ of a Kummer surface.

More generally, by the above argument, for any representation of $Q_8$
in $Aut(\bE^2)$ the numbers $n_8$, $n_4$ and $n_2$ defined above satisfy
the following two equations
$$\begin{array}{r}n_8+2n_4+4n_2=16\\4n_8+3n_4+n_2=19\end{array}$$
Clearly the numbers $n_i$ are non-negative integers and $n_8>0$. We
find out that there are two solutions of this system:
\begin{equation}\label{no_sing_pts}
(n_8,n_4,n_2)=(2,3,2)\ \ {\rm or}\ \ (n_8,n_4,n_2)=(4,0,3).
\end{equation}

The latter solution is satisfied for the following representation
of $Q_8$ in $Aut(\bE^2)$:
$$\begin{array}{ccc}
B_I=\left(\begin{array}{cc}i&0\\i+1&-i\end{array}\right)&
B_J=\left(\begin{array}{cc}-i&i-1\\0&i\end{array}\right)&
B_K=\left(\begin{array}{cc}1&-1-i\\1-i&-1\end{array}\right)
\end{array}$$
The fixed points for this representation are
$(p_0,p_0),\ (p_0,p_1),\ (p_1,p_0),\ (p_1,p_1)$.
Therefore the remaining order two points have isotropy equal to $\bZ_2$.
\medskip

It is convenient to describe the action of the group $Q_8$ in terms of
$\bZ_2[i]$ modules.  By $M^r$ we will denote the free module
$(\bZ_2[i])^{\oplus r}$ and by $M_0^r$ its sub-module $(1+i)\cdot M^r$
which consists of elements annihilated by $1+i$.  As noted above, the
points with non-trivial isotropy with respect to the action of $Q_8$
on $\bE^2$ are of order 2 as $-id$ is contained in every non-trivial
subgroup of $Q_8$. Thus, in fact, in order to understand isotropy of
the action of $Q_8$ on $\bE^2$ we can look into the action of
$Q_8/\langle -id\rangle$ on the set of order two points on $\bE^2$
which is the free $\bZ_2[i]$-module $M^2$. That is, we focus on
representations of $Q_8/\langle -id\rangle=\bZ_2^{\oplus 2}$ in
$SL(2,\bZ_2[i])$.

From this point of view the first representation of $Q_8$ is reduced
to
$$\begin{array}{ccc}
A_I=\left(\begin{array}{cc}i&0\\0&i\end{array}\right)&
A_J=\left(\begin{array}{cc}0&1\\1&0\end{array}\right)&
A_K=\left(\begin{array}{cc}0&i\\i&0\end{array}\right)
\end{array}$$
while the second is
$$\begin{array}{ccc}
B_I=\left(\begin{array}{cc}i&0\\i+1&i\end{array}\right)&
B_J=\left(\begin{array}{cc}i&i+1\\0&i\end{array}\right)&
B_K=\left(\begin{array}{cc}1&i+1\\1+i&1\end{array}\right)
\end{array}$$

Now the claims about the fixed points sets are easy to verify on
$M^2$. Note that the second representation fixes the points whose
coordinates are annihilated by $1+i$ which is the module $M^2_0$.

\subsection{A symplectic Kummer 4-fold}

After discussing the 2-dimensional case we pass to dimension 4.  We
use the notation consistent with the preceding subsection. First, we
prove somewhat more general result on symplectic Kummer 4-folds.

\begin{proposition}\label{Kummer4fold-1}
  Suppose that $G'<SL(4,\bZ[i])$ is a finite subgroup such that
  \begin{itemize}
  \item as a subgroup of $SL(4,\bC)\supset SL(4,\bZ[i])$ $G'$ is
    conjugate to the group $G$ and, in particular, it is generated by
    five order 2 matrices $T_i'$;
  \item the reduction $G'/\langle -id\rangle \ra SL(4,\bZ_2[i])$ acts
    on $M^4$ so that its action on $M^4_0$ is trivial and every
    element in $M^4\setminus M^4_0$ has isotropy generated by $[\pm
    T_i']$ for some $i$.
  \end{itemize}
  Then $G'$ acts on $\bE^4$ and the quotient $\bE^4/G'$ admits a
  symplectic resolution $X$ such that its Betti numbers are as
  follows: $b_2(X)=b_6(X)=23$ and $b_4(X)=276$.
\end{proposition}

\begin{proof}
  Clearly $G'<SL(4,\bZ[i])$ yields an action of $G'$ on $\bE^4$.  We
  claim that any nontrivial isotropy group of this action which is
  different from a symplectic reflection $T_i'$ is actually isomorphic
  and conjugate in $SL(4,\bC)$ to $G$ or to a group $\langle
  T_i',-T_i'\rangle$. This follows from the fact that every subgroup
  of $G$ which is different from $\langle T_i\rangle$ contains $-id$
  hence it is isotropy of a point on $\bE^4$ which is of order 2. Thus
  the isotropy of the action of $G'$ on $\bE^4$ can be understood by
  looking at the action of $G'/\langle -id\rangle$ on the module $M^4$
  and the claim follows.

  Locally both $\bC^4/G$ and $\bC^4/\langle \pm T_i\rangle$ admit
  symplectic resolutions. These resolutions can be glued to the global
  symplectic resolution $X$ of $\bE^4/G'$ because outside isolated
  points which are images of the order 2 points in $\bE^4$ the
  singularities are 2-dimensional families of $A_1$ surface
  singularities whose resolution is unique.

  Thus it remains to calculate the Betti numbers of $X$. For this it
  is enough to find the number of the irreducible 2-dimensional
  components of the singular locus of $\bE^4/G'$. The normalization of
  such components is a quotient of the fixed point set of some $T_i'$
  by the group $N(T_i')/\langle T_i'\rangle=Q_8$ which, as we noted in
  the preceding subsection \ref{no_sing_pts}, has either 2 or 4 points
  of isotropy equal to $Q_8$. On the other hand, $\bE^4/G'$ has
  $16=|M^4_0|$ singular points of type $\bC^4/G$ and each of them
  belongs to 5 different 2-dimensional components of the singular
  locus of $\bE^4/G'$. Since each such component contains 4 such points
  at most it follows that the number of these components is $16\cdot
  5/ 4= 20$ at least. Now the result follows by \cite{Guan}.
\end{proof}

From the above calculations one can conclude that any resolution $X\ra
\bE^4/G'$ contracts 20 exceptional divisors and it has 16 fibers of
type discussed in Section \ref{geom} and 30 fibers which are
isomorphic to $\bP^1\times\bP^1$. 

Indeed, if $G'$ satisfies assumptions of Proposition
\ref{Kummer4fold-1} then $2^4$ out of $2^8$ 2-torsion points in $\bE^4$
have isotropy equal to $G'$ while the remaining $2^8-2^4 = 240$ points
have isotropy equal to the group $\bZ_2\times\bZ_2$ generated by $\pm
T_i$: these subgroups are normal and of index 8 hence in the quotient
we get $240/8 = 30$ points with quotient singularities of type
$\bC^4/\langle \pm T_i\rangle$, the resolution of which is known to
have $\bP^1\times\bP^1$ as the 2-dimensional central fiber.

Although the number of such resolutions is $81^{16}$ (some of these
resolutions lead to non-projective varieties) there exists a unique
special resolution which over each of the 16 points in $\bE^4/G'$ with
singularity of type $\bC^4/G$ is of type $\varphi^\kappa$ described in
\ref{central-resolution}.

\medskip Let us consider the following 5 matrices defined over
$\bZ[i]$:

$$\begin{array}{cc}
T_0'=\left(\begin{array}{cccc}
1&0&0&0\\0&-1&0&0\\0&-1+i&1&0\\1-i&0&0&-1
\end{array}\right)
&
T_1'=\left(\begin{array}{cccc}
i&-1-i&0&1-i\\0&-i&-1+i&0\\0&-1-i&i&0\\1+i&0&-1-i&-i
\end{array}\right)
\end{array}$$
$$\begin{array}{cc}
T_2'=\left(\begin{array}{cccc}
1&0&0&-1-i\\0&-1&1+i&0\\0&0&1&0\\0&0&0&-1
\end{array}\right)
&
T_3'=\left(\begin{array}{cccc}
i&0&0&1-i\\1-i&-i&-1+i&0\\0&-1-i&i&1-i\\1+i&0&0&-i
\end{array}\right)
\end{array}$$
$$T_4'=\left(\begin{array}{cccc}
1&-1+i&0&-1-i\\-1-i&-1&1+i&0\\0&-1+i&1&-1-i\\1-i&0&-1+i&-1
\end{array}\right)$$

\begin{lemma}\label{Kummer4fold-2}
  The group $G'$ generated by matrices $T_i'$ satisfies the assumption
  of Proposition \ref{Kummer4fold-1}.
\end{lemma}

\begin{proof}
  Let us define
  $$ W=\left(\begin{array}{cccc}
      1&1/2-i/2&-1&0\\1&-1/2+i/2&0&-1-i\\0&-1/2-i/2&i&0\\0&1/2+i/2&0&0
    \end{array} \right)$$
  Then for $i=0,\dots,4$ it holds $T_i'=W^{-1}\cdot T_i\cdot W$
  where $T_i$'s are the matrices from the list \ref{matrices}.
  Clearly, the reduction of each of $T_i'$ to $M^4$
  (which by abuse we denote by $T_i'$ as well)
  is identity on $M^4_0$
  and it remains to check that the isotropy of every element from
  $M^4\setminus M^4_0$ is generated by some $T_i'$. This can be done as follows:
  one can write $\ker(T_i'-id)$ as $M^4_0+K_i$ where $K_i$ is a rank 2
  free $\bZ_2[i]$-module. Next one checks that for $i\ne j$ it holds $K_i+K_j=M^4$
  and thus $K_i\cap K_j=\{0\}$. Finally, because $|(M_0^4+K_i)\setminus M_0^4|=48$ it
  follows that $|M^4|=|M_0^4|+\sum_i|(M_0^4+K_i)\setminus M_0^4|$ and therefore every element of
  $M^4\setminus M^4_0$ belongs to exactly one $(M_0^4+K_i)\setminus M^4_0$. We omit calculations.
\end{proof}

\begin{corollary}\label{Kummer4fold-3}
  The quotient $\bE^4/G'$ has a resolution which is a Kummer
  symplectic 4-fold $X$ with  $b_2(X)=b_6(X)=23$ and $b_4(X)=276$.
\end{corollary}


\bibliography{biblio} \bibliographystyle{alpha}

\end{document}